\documentclass[12pt]{article}
\usepackage{graphicx,psfrag,epsf,color}
\usepackage{amssymb}
\usepackage{epsfig, graphicx, amsmath}
\usepackage{natbib}
\usepackage{multirow}
\usepackage{bm, color}
\usepackage[ruled]{algorithm}
\usepackage{algpseudocode}
\usepackage{hyperref}
\hypersetup{
	colorlinks=true,
	linkcolor=red,
	urlcolor=blue,
	citecolor=blue
}
\usepackage{enumerate}
\usepackage{url} % not crucial - just used below for the URL 
\usepackage{array}
\usepackage{booktabs}
\usepackage{float}
\usepackage{comment}
\usepackage[font=small,labelfont=bf]{caption}
\setcitestyle{citesep={;}}
\usepackage[total={6.4in,8.6in}]{geometry}
\newcommand{\mb}[1]{\mathbb{#1}}
\newcommand{\wh}[1]{\widehat{#1}}

\newtheorem{lemma}{Lemma}
\newtheorem{theorem}{Theorem}
\newtheorem{assumption}{Assumption}
\newtheorem{remark}{Remark}

\newtheorem{corollary}{Corollary}
\newenvironment{proof}[1][Proof]{\noindent \textbf{#1.} }{\  \rule{0.5em}{0.5em}}

\newcommand{\blind}{0}

\newcommand{\prob}{{\mbox{Pr}}}

\def\EE{\mathbb{E}}

\def\given{\, | \,}

\def\begar{$$\begin{array}{lll}}
\def\endar{\end{array}$$}
\def\begarlab{\begin{equation} \begin{array}{lll} \label}
\def\endarlab{\end{array} \end{equation}}

\def\ds1{{\mathrm{1 \hspace{-2.6pt} I}}}

\def\calA{{\cal A}}
\def\calB{{\cal B}}

\def\calD{{\cal D}}
\def\calE{{\cal E}}

\def\calI{{\cal I}}

\def\calN{{\cal N}}
\def\calP{{\cal P}}

\def\calR{{\cal R}}

\def\calS{{\cal S}}
\def\calT{{\cal T}}

\def\calX{{\cal X}}

\def\floor#1{\lfloor #1 \rfloor}

\usepackage{xr}

\newcommand{\abs}[1]{|#1|}

\newcommand{\norm}[1]{\|#1\|}
\newcommand{\QZL}[1]{{\color{red}{#1}}}

\begin{document}
	
	\def\spacingset#1{\renewcommand{\baselinestretch}%
		{#1}\small\normalsize} \spacingset{1}
	%
	%
	%%%%%%%%%%%%%%%%%%%%%%%%%%%%%%%%%%%%%%%%%%%%%%%%%%%%%%%%%%%%%%%%%%%%%%%%%%%%%%%
	\if0\blind
	{
		\title{\bf On Well-posedness and Minimax Optimal Rates of Nonparametric $Q$-function Estimation in Off-policy Evaluation\footnotetext{Author
	names are sorted alphabetically}}
		\author{Xiaohong Chen\thanks{Cowles Foundation for Research in Economics, Yale University. \texttt{xiaohong.chen@yale.edu}}
		\quad \quad Zhengling Qi\thanks{Department of Decision Sciences, George Washington University. \texttt{qizhengling@gwu.edu}} 
		}
			
		\maketitle
	} \fi

	\begin{abstract}
		We study the off-policy evaluation (OPE) problem in an infinite-horizon Markov decision process with continuous states and actions. We recast the $Q$-function estimation into a special form of the nonparametric instrumental variables (NPIV) estimation problem.
		We first show that under one mild condition the NPIV formulation of $Q$-function estimation is well-posed in the sense of \textit{$L^2$-measure of ill-posedness} with respect to the data generating distribution, bypassing a strong assumption on the discount factor $\gamma$ imposed in the recent literature for obtaining the $L^2$ convergence rates of various $Q$-function estimators. Thanks to this new well-posed property, we derive the first minimax lower bounds for the convergence rates of nonparametric estimation of $Q$-function and its derivatives in both sup-norm and $L^2$-norm, which are shown to be the same as those for the classical nonparametric regression \citep{stone1982optimal}. We then propose a sieve two-stage least squares estimator and establish its rate-optimality in both norms under some mild conditions.
		%We show that the (in sup-norm and $L^2$-norm). of a sieve 2SLS estimator for the $Q$-function and its derivatives. The %convergence rates , which	are also the same as the well-known minimax optimal rates for a nonparametric regression \citep{stone1982optimal}. 
		%achieve the  we establish sup-norm rates on the proposed $Q$-function estimator and its derivatives and show they can %achieve the minimax-optimal sup-norm rates, which are the same as those in nonparametric regression. 
		Our general results on the well-posedness and the minimax lower bounds are of independent interest to study not only other nonparametric estimators for $Q$-function but also efficient estimation on the value of any target policy in off-policy settings. 
		%To the best of our knowledge, this is the first minimax lower bound result for nonparametric off-policy $Q$-function %estimation in an infinite-horizon Markov setting (with continuous states). 
		%As a byproduct of the well-posedness, we also establish lower bounds in terms of $L^2$-norm.
		
	\end{abstract}

	\baselineskip=21pt
	\section{Introduction}
	In recent years, there is a surging interest in studying batch reinforcement learning (RL), which utilizes previously collected data to perform sequential decision making \citep{sutton2018reinforcement} and 
	does not require interacting with task environment or accessing a simulator.
	%which is often not possible in some high-stake real-world applications such as health care. 
	The batch RL techniques are especially attractive in many high-stake real-world application domains where it is too costly or infeasible to access a simulator, such as
	%since they 
	%While batch RL has already demonstrated its potentials in many domains such as 
	mobile health \citep{liao2018just}, robotics \citep{pinto2016supersizing}, digital marketing \citep{thomas2017predictive} and precision medicine \citep{kosorok2019precision}, and others.
	Nevertheless, the batch setting still posits several theoretical challenges that tamper the generalizability of many RL algorithms in practice. Among them, one central challenge is the \textit{distributional mismatch} between the data collecting process and the target distribution for evaluation \citep{levine2020offline}.
	
	Motivated by these,  we study the off-policy evaluation (OPE) problem, which is considered one of fundamental problems in batch RL. The goal of OPE is to leverage pre-collected data generated by a so-called behavior policy to evaluate the performance (e.g., value) of a new/target policy. In particular, we investigate theoretical property of nonparametric estimation of $Q$-function in the setting of infinite-horizon Markov decision processes (MDPs) (with discounted rewards, continuous states and actions). 
	
	We make several important contributions to the existing literature. Motivated by Bellman equation, we formulate $Q$-function estimation under the framework of a nonparametric instrumental variable (NPIV) model. 
	%The NPIV model, which has been well developed in econometric literature, is suitable for analyzing the theoretical property of %$Q$-function estimation. 
	We first show that, under mild regularity conditions, the NPIV formulation of $Q$-function estimation is well-posed in the sense of \textit{$L^2$-measure of ill-posedness} with respect to the data generating distribution. This essentially justifies the valid use of the $L^2$-norm of Bellman error/residual to measure the accuracy of $Q$-function estimation in the batch setting. Next, we derive the minimax lower bounds for the convergence rates in sup-norm and in $L^2$-norm for the estimation of $Q$-function and its derivatives. Thanks to the general well-posedness result, the lower bounds are shown to be the same as those for the nonparametric regression estimation in the i.i.d. setting \citep{stone1982optimal,Tsybakov2009}. Thus the nonparametric $Q$-function estimation could be as easy as the nonparametric regression in terms of the worst case rate. Using the NPIV formulation, we also propose sieve 2SLS estimators to estimate the $Q$-function (and its derivatives) and establish their convergence rates in both sup-norm and $L^2$-norm. In particular, B-spline and wavelet 2SLS estimators are shown to achieve the sup-norm lower bound for H\"older class of $Q$-functions (and the derivatives), and many more linear sieve (such as polynomials, cosines, splines, wavelets) 2SLS estimators are shown to achieve the $L^2$-norm lower bound for Sobolev class of $Q$-functions (and the derivatives). Our results on $L^2$-norm convergence rates under mild conditions are particularly useful for obtaining efficient estimation and optimal inference on the value (i.e., the expectation of the $Q$ function) of a target policy. 
	%Under mild regularity conditions we show that the convergence rates of our sieve 2SLS estimators for the $Q$-function and its derivatives coincide with the lower bounds in both sup-norm and in $L^2$-norm, and hence achieve the same minimax optimal rates as those for a nonparametric regression \citep{stone1982optimal}.  
To the best of our knowledge, ours are the first minimax results for non-parametrically estimating $Q$-function of continuous states and actions in the off-policy setting. 
The general results on the well-posedness and the minimax lower bounds (in sup-norm and in $L^2$-norm)
	%	%(in sup-norm and in $L^2$-norm) and the optimal sup-norm convergence rates 
		are of independent interest to study properties of other nonparametric estimators for $Q$-function and the related estimators of the marginal importance weight (see, e.g., \cite{liu2018breaking}) in the off-policy setting. 
	
	   %Interestingly, there is a very recent work by \citep{wang2021projected} conjecturing that there may exist a phase transition pattern in estimating both ratio and $Q$-function, i.e., the problem could be mildly ill-posed even under the coverage assumption (See Proposition 1 of \cite{wang2021projected} for the generic lower bound for the minimum eigenvalue).  We show that such pattern does not exist under some mild regularity condition. 
	  %Based on the well-posedness result, we establish a sup-norm rate on the proposed $Q$-function estimator and its derivatives and show that they enjoy the minimax-optimality. To the best of our knowledge, this is the first minimax result for non-parametrically estimating $Q$-function and its derivatives in the offline setting. As of independent interest, we also establish the minimax lower bound for estimating $Q$-function and its derivatives in $L^2$-norm, which seems a new addition to the existing literature as well.	Our results on the well-posedness and the minimax lower bounds (in sup-norm and in $L^2$-norm)
		%(in sup-norm and in $L^2$-norm) and the optimal sup-norm convergence rates 
		%are of independent interest to study properties of other nonparametric estimators for $Q$-function in the off-policy setting (with continuous states). 
	\subsection{Closely Related Work}
	Estimation of $Q$-function for a fixed policy is a key building block for many RL algorithms. There is a growing literature on nonparametric estimation of $Q$-function in the infinite-horizon and off-policy setting. See some recent theoretical development in \cite{farahmand2016regularized,shi2020statistical,uehara2021finite} among many others. Specifically, \cite{farahmand2016regularized} established $L^2$ error bound for Bellman error of their $Q$-function estimator.  \cite{shi2020statistical,uehara2021finite} derived that $L^2$-norm convergence rates and error bounds for their respective nonparametric $Q$-function estimators under 
	%to obtain the $L^2$ convergence rates or bounds of their $Q$-function estimators, 
	a strong assumption that is essentially equivalent to restricting the discount factor $\gamma$ to be close to zero.
	%is imposed, together with a so-called \textit{coverage} assumption, so that $L^2$ error of any $Q$-function estimator is bounded %above by the corresponding Bellman error multiplied by some constant factor. %However, restricting the discount factor $\gamma$ %limits their OPE methods from evaluating a target policy that is very different from the behavior one, therefore weakening the %generalizability of their methods. 
	Our well-posedness result implies that their $L^2$-norm convergence rates of their respective estimators for $Q$-function remain valid without their strong assumption on the discount rate $\gamma$.
	%beyound its natural range of $[0,1)$. 
	%Therefore we justify the use of $L^2$-norm to measure Bellman error of $Q$-function estimation under some  mild condition. 
	See Section \ref{sec: well-posed} and Remark \ref{rm:Shi} for more detailed discussions.

	 The connection of estimating $Q$-function in Bellman equation to instrumental variables estimation, to the best of our knowledge, has been first pointed out by  \cite{bradtke1996linear} for their celebrated least-squares temporal difference (LSTD) method for parametric models. Recently, the relation between nonparametric $Q$-function estimation and nonparametric instrumental variables (NPIV) estimation has also been observed by some applied work (such as \cite{chen2021instrumental}) and theoretical work (such as \cite{duan2021optimal} that focuses on the on-policy setting). The NPIV model has been extensively investigated in econometric literature; see, e.g., 
	\cite{newey2003instrumental,ai2003efficient,HH2005,blundell2007semi,DFFR2011,chen2011rate,chen2013optimal} for earlier reference. However, there is some subtle difference between the nonparametric $Q$-function estimation and the NPIV one. It is known that a generic NPIV model with continuous endogenous variables is a difficult ill-posed inverse problem in econometrics, but we show that estimation of a nonparametric $Q$-function of continuous states and actions can be well-posed under mild regularity conditions that are typically assumed in batch RL literature. Our well-posedness result implies that nonparametric estimation and inference on OPE and related batch RL problems  could be much simpler than the difficult ill-posed NPIV problems studied in the existing econometric literature.

	The rest of the paper is organized as follows. Section \ref{sec: prelim} presents the framework of infinite-horizon MDPs and some necessary notations. In Section \ref{sec: well-posed}, we show that the nonparametric $Q$-function estimation in sup-norm and in $L^2$-norm are both well-posed. Section \ref{sec: lower bound} establishes the minimax lower bounds for the rates of convergence for nonparametric estimation of $Q$-function in sup-norm and in $L^2$-norm respectively.
	%, which are the same as the well-known lower bounds established for nonparametric regression for i.i.d. data. 
	In Section \ref{sec: NPIV}, we propose sieve 2SLS estimation of the $Q$-function and its derivatives. Under some mild condition, we establish their rates of convergence in both sup-norm and $L^2$-norm, which coincide with the lower bounds. Section \ref{sec: conclusion} briefly concludes. Most proofs are given in the appendix.
	
	%
	%ranging from statistics, econometrics, operations research and computer science. In statistical literature, OPE problem is similar to causal inference for longitudinal data
	
	%\textbf{Main Contribution}
	%\begin{itemize}
	%	\item Well-posedness under NPIV regression
	%	\item Sup-norm rate
	%	\item Lower bound
	%\end{itemize}
	
	\section{Preliminaries and Notation}\label{sec: prelim}
	
	Consider a single trajectory $\{(S_{t},A_{t},R_{t})\}_{t \geq 0}$ where $(S_{t},A_{t},R_{t})$ denotes the state-action-reward triplet collected at time $t$. Let $\calS$ and $\calA$ be the state and action spaces, respectively. We assume both state and action are \textit{continuous} (as the discrete and finite spaces are easier). 
	%For simplicity we assume that the rewards $R_t$ are uniformly bounded for $t \geq 0$. 
	A policy associated with this trajectory defines an agent's way of choosing the action at each decision time $t$. In this paper, we focus on using the batch data to evaluate the performance of a stationary policy denoted by $\pi$, which is a function mapping from the state space $\calS$ to a probability distribution over $\calA$. In particular, $\pi(a \given s)$ refers to the probability density function of choosing action $a \in \calA$ given the state value $s \in \calS$. In addition, let $\calS \times \calA \subseteq \mathbb{R}^{d}$ for some $d \geq 2$, and $\calB(\calS)$ be the family of Borel subsets of $\calS$. 
	
	The main goal of this paper is to estimate the so-called $Q$-function of a target policy $\pi$ using the batch data. Specifically, given a stationary policy $\pi$ and any state-action pair $(s, a) \in \calS \times \calA$, we define $Q$-function as
	\begin{eqnarray*}
		Q^{\pi}(s,a)=\sum_{t=0}^{+\infty} \gamma^t \EE^{\pi} (R_{t}|S_{0}=s, A_{0}=a),
	\end{eqnarray*}
	where $\EE^{\pi}$ denotes the expectation assuming the actions are selected according to $\pi$, and $0 \leq \gamma<1$ denotes some discounted factor that balances the trade-off between immediate and future rewards. We consider the framework of a time-homogeneous MDP and hence make the following two assumptions, which are foundation of many $Q$-function estimations.
	\begin{assumption}\label{ass: Markovian}
		There exists a transition kernel $P$ such that for every $t\geq 1$, $s \in \calS$, $a \in \calA$ and any set $B \in \calB(\calS)$,
		$$
		\Pr(S_{t+1}  \in B \given S_t = s, A_t = a, \left\{S_j, A_j, R_j\right\}_{0 \leq j < t}) = P(S_{t+1}  \in B \given S_t = s, A_t = a),
		$$
		In addition, there exists a probability density function $q$ for the transition kernel $P$.
	\end{assumption}
	\begin{assumption}\label{ass: reward}
	    For every $t \geq 0$, $R_t = \calR(S_t, A_t, S_{t+1})$, i.e., a measurable function of $(S_t, A_t, S_{t+1})$.
		In addition, there exists a finite constant $R_{\max}$ such that $\abs{R_t} \leq R_{\max}$ for all $t \geq0$.
	\end{assumption}
	 Let $r(s, a) = \EE\left[R_{t} \given S_{t} = s, A_{t} = a\right]$ for every $t \geq 0$, $s \in \calS$ and $a \in \calA$. Assumption \ref{ass: reward} implies that $\abs{r(S_t, A_t)} \leq R_{\max}$ for all $t \geq0$. 
	 We note that the uniformly bounded reward assumption is imposed for simplicity only, and can be replaced by assuming existence of higher order conditional moments of $R_t$ given $(S_t,A_t)$; see, e.g., \cite{chen2015optimal,chen2013optimal}.

	To estimate $Q^\pi$, by Assumptions \ref{ass: Markovian} and \ref{ass: reward}, one approach is to solve the following Bellman equation, i.e.,
	\begin{align}\label{eq: Bellman equation for Q}
	Q^\pi(s, a) = \EE\left[R_t + \gamma \int_{a' \in \calA} \pi(a' \given S_{t+1})Q^\pi(S_{t+1}, a')\text{d}a'  \given S_t = s, A_t = a\right],
	\end{align}
	for any $t\geq 0$, $s \in \calS $ and $a \in \calA$. Throughout this paper, we assume the integration with respect to $\pi$ in \eqref{eq: Bellman equation for Q} can be exactly evaluated as long as the integrand is known. In practice, one can use Monte Carlo method to approximate this integration since the target policy $\pi$ is known. 
	
	Now suppose the given batch data consist of $N$ trajectories, which correspond to $N$ independent and identically distributed copies of $\{(S_{t},A_{t},R_{t})\}_{t\ge 0}$. For $1 \leq i \leq N$, data collected from the $i$th trajectory are represented by $\{(S_{i,t},A_{i,t},R_{i,t},S_{i,t+1})\}_{0\le t< T}$. We then aim to leverage this batch data to estimate $Q$-function of a target policy $\pi$. Before presenting our theoretical results and methods, we make one additional assumption on the data generating process. Let $\pi^b$ be a stationary policy and $\pi^b(a \given s)$ refers to the conditional probability density of choosing the action $a$ given the state value $s$.

	\begin{assumption}\label{ass: DGP}
		The batch data $\calD_N = \{(S_{i,t},A_{i,t},R_{i,t},S_{i,t+1})\}_{0\le t< T, 1 \le i \le N}$ are generated by the policy $\pi^b$. %In particular, $\pi^b(a \given s)$ refers to the conditional probability density of choosing the action $a$ given the state value $s$.
	\end{assumption}
	
	Assumptions \ref{ass: Markovian}-\ref{ass: DGP} are standard in the literature of batch RL. Note that in the literature the policy $\pi^b$ is  often called the behavior policy and mostly different from the target one $\pi$. Next, we introduce the average visitation probability measure. Let $q^{\pi^b}_t(s, a)$ be the marginal probability density of a state-action pair $(s, a)$ at the decision point $t$ induced by the behavior policy $\pi^b$. Then the average visitation probability density across $T$ decision points is defined as
	$$
	\bar d^{\pi^b}_T(s, a) = \frac{1}{T}\sum_{t = 0}^{T-1}q^{\pi^b}_t(s, a).
	$$
	The corresponding expectation with respect to $\bar d^{\pi^b}_T$ is denoted by $\overline \EE$. We further let $q^\pi_t(s', a' \given s, a)$
	be the $t$-step visitation probability density function induced by a policy $\pi$ at $(s', a')$ given an initial state-action pair $(s, a) \in \calS \times \calA$.
	
	\textit{Notation}: For generic sequences $\{\varpi(N)\}_{N\geq1}$ and $\{\theta(N)\}_{N\geq1}$, the notation $\varpi(N) \gtrsim  \theta(N)$ (resp. $\varpi(N) \lesssim \theta(N)$) means that there exists a sufficiently large constant (resp. small) constant $c_1>0$ (resp. $c_2>0$) such that $\varpi(N) \geq c_1 \theta(N)$ (resp. $\varpi(N) \leq c_2 \theta(N)$). We use $\varpi(N) \asymp \theta(N)$ when $\varpi(N) \gtrsim \theta(N)$ and $\varpi(N) \lesssim \theta(N)$. For matrix and vector norms, we use $\norm{\bullet}_{\ell_q}$ to denote  either  the vector $\ell_q$-norm or operator norm induced by the vector $\ell_q$-norm, for $1 \leq q < \infty$, when there is no confusion. $\lambda_{\min}(\bullet)$ and $\lambda_{\max}(\bullet)$  denote the minimum and maximum eigenvalues of some square matrix, respectively.  For any random variable $X$, we use $L^q(X)$ to denote the class of all measurable functions with finite $q$-th moments for $1 \leq q \leq \infty$. Then the $L^q$-norm is denoted by $\norm{\bullet}_{L^q(X)}$. When there is no confusion in the underlying distribution, we also write it as $\norm{\bullet}_{L^q}$ or $\norm{\bullet}_{q}$.  In particular, $\norm{\bullet}_\infty$ denotes the sup-norm. In addition, we use Big $O_p$ and small $o_p$ as the convention. We often use $(S, A, R, S')$ or $(S, A, S')$ to represent some generic transition tuples, where the transition probability density is $q$.  Lastly, we introduce the H\"older class of functions $g: \calX \subseteq \mathbb{R}^d \rightarrow \mathbb{R}$ with smoothness $p > 0$ as
	\begin{align*}\label{def: holder class}
	&\Lambda_\infty(p, L) \triangleq & \left\{g \quad  \mid \sup_{0 \leq \norm{\alpha}_{\ell_1} \leq \floor p} \norm{\partial^\alpha g}_\infty \leq L,  \quad \sup_{\alpha: \norm{\alpha}_{\ell_1} = \floor p} \sup_{x, y \in \calX,  x \neq y} \frac{\left|\partial^\alpha g(x) - \partial^\alpha g(y)\right|}{\norm{x - y}_{\ell_2}^{\alpha-\floor p}} \leq L \right\},
	\end{align*}
	where $\calX = \calS \times \calA \subset \mathbb{R}^{d}$ is a compact rectangular support with nonempty interior, $\floor p$ denotes the integer no larger than $p$ for any $p >0$, a non-negative vector $\alpha = (\alpha_1, \alpha_2, \cdots, \alpha_d)$ and
	$$
	\partial^\alpha g(x) = \frac{\partial^\alpha g(x) }{\partial x_1^{\alpha_1} \partial x_2^{\alpha_2}\cdots \partial x_d^{\alpha_d}}.
	$$
	We let $\Lambda_2(p, L)$ be the Sobolev space of smoothness $p$ with radius $L$ and support $\calX$, where the underlying measure is Lebesque measure.

	\section{A Special Form of NPIV Models: Well-posedness}\label{sec: well-posed}
	
	In this section, we formulate $Q$-function estimation under the framework of a nonparametric instrumental variables (NPIV) model, which has been extensively studied in econometrics (e.g., \cite{ai2003efficient,newey2003instrumental,blundell2007semi}). A generic NPIV model takes the expression as
	\begin{equation}\label{Model:NPIV}
	Y = h_0(X) + U, \quad \text{with}  \quad \EE[U | W] = 0,
	\end{equation}
	where $h_0$ is an unknown function to estimate, $X$ is called endogenous variables, $W$ is called instrumental variables, and $U$ represents some random error. 
	%The basis of NPIV estimation is the conditional moment restriction, i.e.,
	%$
	%\EE\left[Y - h_0(X) \given W\right] = 0,
	%$ implied by Model \eqref{Model:NPIV}. 
	Motivated by Equation \eqref{eq: Bellman equation for Q}, we consider the following special form of a NPIV model with Assumptions \ref{ass: Markovian}-\ref{ass: DGP} for $Q$-function estimation:
	\begin{align}\label{Model: NPIV on $Q$-function}
	R_t = h^\pi(S_t, A_t, S_{t+1};Q^\pi) + U_t, \quad \text{with}  \quad \EE[U_t | S_t, A_t] = 0
	\end{align}
	for $0 \leq t \leq T-1$, where 
	\[h^\pi(s, a, s';Q) = Q(s, a) - \gamma\int_{a' \in \calA}\pi(a' | s') Q(s', a')\text{d}a'.
	\]
	We also write $h^\pi(s, a, s';Q)$ as  $h^\pi(Q)(s, a, s')$ and $h_0^\pi = h^\pi(Q^\pi)$ when there is no confusion. By requiring $\EE[U_t | S_t, A_t] = 0$ for $0 \leq t \leq T-1$, we recover the Bellman Equation \eqref{eq: Bellman equation for Q}. Therefore Model \eqref{Model: NPIV on $Q$-function} can be used to estimate $Q^\pi$ nonparametrically, where $S_{t+1}$ can be understood as endogenous variables and $(S_t, A_t)$ as instrumental variables under the framework of the NPIV model. 
	Let $L^2(S, A)$ be the space of square integrable functions against the probability measure with density $\bar d^{\pi^b}_T$ and $L^2(S, A, S')$ against the probability measure with density $\bar d^{\pi^b}_T \times q$. 
	Denote the conditional expectation operator by $\calT: L^2(S, A, S') \to L^2(S, A)$, i.e., for every $(s, a) \in \calS \times \calA$,
	\[
	\calT f(s, a) = \EE[f(S, A, S')|S = s, A = a]\,
	\]
	and in particular,
	\[
	\calT h^\pi(Q)(s,a) = \EE[h^\pi(S, A, S'; Q)|S = s, A = a]\,.
	\]
	
	%----------------------

	 %-------------------------------------------------
	 \subsection{Well-posedness in sup-norm}
	 In this subsection, we show that $Q$-function estimation is in general well-posed in sup-norm given by the following lemma.
	 
	\begin{lemma}\label{lm: contraction}
		For any discount factor $0 \leq \gamma < 1$ and any uniformly bounded function $Q$ defined over $(\calS, \calA)$, the following inequalities hold.
		\begin{align}\label{eq: link Q to Bellman operator to TD}
		\frac{1}{1+\gamma }\norm{h^\pi(Q- Q^\pi)}_\infty \leq \norm{Q - Q^\pi}_\infty \leq \frac{1}{1-\gamma}\norm{\calT h^\pi(Q - Q^\pi)}_\infty \leq \frac{1}{1-\gamma}\norm{h^\pi(Q- Q^\pi)}_\infty.
		\end{align}
	\end{lemma}
	\begin{proof}
		It is sufficient to show that $\norm{Q - Q^\pi}_\infty \leq \frac{1}{1-\gamma}\norm{\calT h^\pi(Q - Q^\pi)}_\infty$, while other inequalities can be readily seen. It can be observed that
		\begin{align}
		\norm{Q - Q^\pi}_\infty & \leq \norm{\calT h^\pi(Q - Q^\pi)}_\infty + \gamma \norm{\EE^\pi\left[(Q- Q^\pi)(S', A') \given S = \bullet, A = \bullet \right]}_\infty \label{eqn: temp1}\\[0.1in]
		& \leq \norm{\calT h^\pi(Q - Q^\pi)}_\infty + \gamma\norm{Q- Q^\pi}_\infty \label{eqn: temp2},
		\end{align}
		where the first line follows the triangle inequality. This immediately implies
		\begin{align}\label{eqn: temp inequality}
		\norm{Q - Q^\pi}_\infty & \leq  \frac{1}{1-\gamma}\norm{\calT h^\pi(Q - Q^\pi)}_\infty.
		\end{align}
	\end{proof}
	
	Lemma \ref{lm: contraction} implies that to obtain the sup-norm rate for $\wh Q^\pi$, it is sufficient to focus on $
	\norm{h^\pi(\wh Q^\pi- Q^\pi)}_\infty$, which is the sup-norm of so-called \textit{temporal difference} error. One key reason of having such an inequality is the fact that Bellman operator is $\gamma$-contractive with respect to the sup-norm (from \eqref{eqn: temp1}-\eqref{eqn: temp2}). However, it is hard to develop an estimator that minimizes the sup-norm of Bellman error in the batch setting so as to directly bound the sup-norm. Instead most existing methods are focused on minimizing the $L^2$-norm of Bellman error.  This motivates us to study the well-posedness in $L^2$-norm below.
	
	\subsection{Well-posedness in $L^2$-norm}\label{subsec: well-posed in L2}
	 
	Lemma \ref{lm: contraction} in general may not hold for $L^2$-norm with respect to the data generating process (e.g., $\bar d^{\pi^b}_T$) due to the distributional mismatch between the behavior policy and the target one, which is one fundamental barrier in analyzing OPE problem in the literature as discussed in the introduction.
	To characterize the difficulty of $L^2$-estimating $Q^\pi$ under Model \eqref{Model: NPIV on $Q$-function}, we define a \textit{$L^2$-measure of ill-posedness} as
	 \begin{equation}\label{def: Lp measure of ill-posedness}
	 \overline \tau  =  \sup_{Q \in L^2(S, A)} \frac{\|h^\pi(Q)\|_{L^2(S, A, S')}}{\|\calT h^\pi(Q)\|_{L^2(S, A)}}\,.
	 \end{equation}
	 It can be seen that $\overline \tau \geq 1$ and could be arbitrarily large in general, which can be used to quantify the level of ill-posedness in estimating $Q^\pi$.  
	 %We remark that $\overline \tau_p$ is general and independent of any estimation method. Interestingly, this measure is closely %related to the so-called modulus of continuity defined in \citep{chen2012estimation}, which is used  to quantify the local %ill-posedness around the true function of the general conditional moment restrictions. Our defined measure $\overline \tau_p$ %can be interpreted as a global version of their modulus of continuity. 
	% In our Section \ref{subsec: well-posed in L2}, we show that $\overline \tau$ is bounded from the above under some mild %condition.
	 We impose the following mild assumption to ensure the well-posedness in $L^2$-norm 
	 in the sense of  $\overline \tau \lesssim 1$.

	\begin{assumption}\label{ass: L2 well-posed}
		%(a)~The stochastic process $\{S_t, A_t\}_{t\geq 0}$ induced by the behavior policy $\pi^b$ is a stationary, exponentially $\boldsymbol{\beta}$-mixing stochastic process. The $\boldsymbol{\beta}$-mixing coefficient at time lag $k$ satisfies that $\beta_k \leq \beta_0 \exp(-\beta_1 k)$ for $\beta_0 \geq 0$ and $\beta_1 > 0$. $R_t = \calR(S_t, A_t, S_{t+1})$ for some function $\calR$ and every $t \geq 0$. The induced stationary density is denoted by $d^{\pi^b}$.
		%The behavior policy is stationary depending on the current state, denoted by $\pi^b$.  
		(a)~There exist positive constants $p_{\min}$ and $p_{1,\max}$ such that the average visitation probability density function $\bar d^{\pi^b}_T$ satisfies $p_{\min} \leq \bar d^{\pi^b}_T(s, a)  \leq p_{1, \max} $ for every $(s, a) \in \calS \times \calA$. %$\calS \times \calA \subset \mathbb{R}^{d}$ is a compact rectangular support with nonempty interior. 
		(b)~ The target policy $\pi$ is absolutely continuous with respect to $\pi^b$ and $q^\pi(s', a' \given s, a) \leq p_{2, \max}$ for some positive constant $p_{2, \max}$.
		%In addition, let $c_0 = \sup_{s \in \calS, a \in \calA}\frac{\pi(a \mid s)}{\pi^b(a \mid s)}  > 0$ and assume $\gamma^2c_0 < 1$.	
	\end{assumption}
	Let $p_{\max} = \max(p_{1, \max}, p_{2, \max})$. 
	%Assumption \ref{ass: stationary}~(a)  characterizes the dependency among observations over time. The  $\boldsymbol{\beta}$-mixing coefficient at time lag $k$ basically means that the dependency between $\{S_t, A_t\}_{t \leq j }$	and $\{S_t, A_t\}_{t \geq (j+k) }$ decays to 0 at an exponential rate with respect to $k$. See \citep{bradley2005basic} for the exact definition of the exponentially $\boldsymbol{\beta}$-mixing. A fast mixing rate is imposed here mainly for technical simplicity. Indeed Assumption \ref{ass: stationary}~(a) could possibly be relaxed to either stationary with algebraic $\beta$-mixing or exponentially $\beta$-mixing without stationary assumption. Since this is not the focus of our paper, we do not further study it but rather keep this relatively strong assumption. 
	%Under Assumption \ref{ass: stationary}~(a), the average visitation probability density $\bar d^{\pi^b}_T$ used in Assumption \ref{ass: stationary}~(b)   becomes the induced stationary density $d^{\pi^b}$. We make them distinctive mainly for the theoretical presentation below. We will omit $\nu$ in $\norm{\bullet}_{2, \nu}$ when $\nu = d^{\pi^b}$. 
	In general, boundedness assumption on the data generating probability density in Assumption \ref{ass: L2 well-posed}~(a)
	is standard in the classical non-parametric estimation such as \cite{huang1998projection,chen2015optimal}. In our setting, that the average visitation probability density is uniformly bounded away from $0$ is also called \textit{coverage} assumption frequently used in RL literature such as \cite{precup2000eligibility,antos2008fitted,kallus2019efficiently} among many others. We use this standard assumption as we consider any target policy for OPE. This assumption can be relaxed to the so-called \textit{partial coverage} if one is willing to impose some structure assumption on $Q^\pi$. See recent studies in \cite{duan2020minimax,xie2021bellman,agarwal2021theory,uehara2021pessimistic}. Assumption \ref{ass: L2 well-posed}~(b) imposes one mild identification condition on the target policy. It essentially states that our batch data are able to identify the value of the target policy.
	Lastly, we remark that when both $\calS$ and $\calA$ are discrete and finite, Assumption \ref{ass: L2 well-posed}~(b) is automatically satisfied because of Assumption \ref{ass: L2 well-posed}~(a).
	%(i.e., with the coverage assumption inequalities in Lemma \ref{lm: contraction} hold in the sense of $L^2$-norm. It is %important to note that when $\calA$ is discrete, 
	In the following, we use $\norm{\bullet}_{2, \nu}$ to denote $L^2$-norm with respect to some probability distribution/density $\nu$.

	Now we are ready to present a key theorem in this paper, which can not only be used to establish the minimax-optimal sup-norm and $L^2$-norm rates for estimating $Q^\pi$, but also provide a foundation for many existing OPE estimators.
	\begin{theorem}\label{thm: well-posed}
		For any policy $\pi$, discount factor $0 \leq \gamma < 1$, and any two square integrable functions $Q_1$ and $Q_2$ defined over $(\calS, \calA)$ with respect to $\bar d_T^{\pi^b}$, under Assumptions \ref{ass: Markovian}, \ref{ass: DGP} and \ref{ass: L2 well-posed}, the following inequalities hold.
		\begin{align}\label{eq: L2 error bound}
		\sqrt{\frac{p_{\min}}{p_{\max}}}(1-\gamma)\norm{Q_1 - Q_2}_{2, \bar d_T^{\pi^b}} \leq \norm{\calT h^\pi(Q_1 - Q_2)}_{2, \bar d_T^{\pi^b}} \leq \norm{h^\pi(Q_1- Q_2)}_{2, \bar d_T^{\pi^b}\times q}.
		\end{align}
		In particular, 
		the $L^2$ measure of ill-posedness
		$$
		\overline \tau \lesssim \frac{\sqrt{p_{\max}(1+\frac{p_{\max}\gamma^2}{p_{\min}})}}{\sqrt {p_{\min}}(1-\gamma)} \lesssim 1.
		$$
	\end{theorem}
	
	\begin{proof}
		For the first statement of Theorem \ref{thm: well-posed}, it is enough to focus on the first inequality, while the second one is given by Jensen's inequality. Let $\calI$ be the identity operator and $\calP^\pi$ be the operator such that $\calP^\pi f (s, a) = \EE^\pi\left[f(S_{t+1}, A_{t+1}) \given S_t = s, A_t = a \right]$ for any $t \geq 0$.
		%Let $\calT h^\pi(Q) = \calI - \gamma \calP^\pi$, where $\calI$ is the identity operator and $\calP^\pi$ is the operator such that $\calP^\pi f (s, a) = \EE^\pi\left[f(S_{t+1}, A_{t+1}) \given S_t = s, A_t = a \right]$ for any $t \geq 0$.
		By induction, we can show that $(\calP^\pi)^{k} f (s, a) = \EE^\pi\left[f(S_{t+k}, A_{t+k}) \given S_t = s, A_t = a \right]$. For some integer $\bar t$, which will be specified later, we have
		\begin{align*}%\label{eqn: contraction}
		\norm{Q_1 - Q_2}_{2, \bar d_T^{\pi^b}} \leq \underbrace{\norm{\left(\calI - \gamma^{\bar t} (\calP^\pi)^{\bar t}\right)\left(Q_1 - Q_2\right)}_{2, \bar d_T^{\pi^b}}}_{(I)} + \gamma^{\bar t} \underbrace{\norm{ (\calP^\pi)^{\bar t}\left(Q_1 - Q_2\right)}_{2, \bar d_T^{\pi^b}}}_{(II)}.
		\end{align*}
		We first focus on deriving an upper bound for $(II)$. By Jensen's inequality, we can show that
		\begin{align*}
		\{(II)\}^2 & \leq \int_{s \in \calS, a \in \calA} \EE^\pi\left[(Q_1 - Q_2)^2(S_{\bar t}, A_{\bar t}) \given S_0 = s, A_0 = a\right] \bar d_T^{\pi^b}(s, a) \text{d}s \text{d}a \\
		& = \int_{s \in \calS, a \in \calA} \int_{s' \in \calS, a' \in \calA} (Q_1 - Q_2)^2(s', a')q^\pi_{\bar t}(s', a' \given s, a) \text{d}s' \text{d}a' \bar d_T^{\pi^b}(s, a) \text{d}s \text{d}a\\
		& = \int_{s' \in \calS, a' \in \calA} (Q_1 - Q_2)^2(s', a') \widetilde{q}^{\pi^b;\pi}_{T;{\bar t}}(s', a')\text{d}s' \text{d}a'\\
		& = \int_{s' \in \calS, a' \in \calA} (Q_1 - Q_2)^2(s', a') \frac{\widetilde{q}^{\pi^b;\pi}_{T;{\bar t}}(s', a')}{\bar d_T^{\pi^b}(s', a')}\bar d_T^{\pi^b}(s', a')\text{d}s' \text{d}a'\\
		& \leq \frac{p_{\max}}{p_{\min}}\norm{Q_1 - Q_2}_{2, \bar d_T^{\pi^b}}^2,
		\end{align*}
		where $\widetilde{q}^{\pi^b;\pi}_{T;{\bar t}}(s', a')$ refers to the marginal probability density function by composition between $\bar d_T^{\pi^b}$ and $q_{\bar t}^\pi$. The last equation holds because $\widetilde{q}^{\pi^b;\pi}_{T;{\bar t}}$ is absolutely continuous with respect to $\bar d_T^{\pi^b}(s', a')$ by Assumption \ref{ass: L2 well-posed}. The last inequality is also given by Assumption \ref{ass: L2 well-posed} since $\widetilde{q}^{\pi^b;\pi}_{T;{\bar t}}(s', a') = \overline{\EE}\left[q^\pi_{\bar t}(s', a' \given S, A)\right] \leq p_{\max}$ for every $(s, a) \in \calS \times \calA$ (As long as one-step transition density is bounded above, $\bar{t}$-step will also be bounded above.). Now for any $\varepsilon > 0$, we can choose $\bar t$ sufficiently large such that 
		$$
		\gamma^{\bar t} \sqrt{p_{\max}/p_{\min}} \leq \varepsilon,
		$$
		which implies that $\gamma^{\bar t} \times (II) \leq \varepsilon \norm{Q_1 - Q_2}_{2, \bar d_T^{\pi^b}}$. This further shows that
		\begin{align*}
		\norm{Q_1 - Q_2}_{2, \bar d_T^{\pi^b}} \leq (1-\varepsilon)^{-1} \times (I).
		\end{align*}
		In the following, we derive an upper bound for $(I)$. Let $g = (\calI - \gamma \calP^\pi)(Q_1 - Q_2)$. By a similar argument as before, we have
		\begin{align*}
		(I) & =  \norm{\left(\calI - \gamma \calP^\pi +\gamma \calP^\pi - \gamma^2 (\calP^\pi)^2 + \cdots + \gamma^{\bar t-1}(\calP^\pi)^{\bar t-1} -  \gamma^{\bar t} (\calP^\pi)^{\bar t}\right)\left(Q_1 - Q_2\right)}_{2, \bar d_T^{\pi^b}}\\
		& \leq \sum_{k = 0}^{\bar t-1}\gamma^{k}\norm{(\calP^\pi)^{k}(\calI - \gamma \calP^\pi)(Q_1 - Q_2)}_{2, \bar d_T^{\pi^b}}\\
		& = \sum_{k = 0}^{\bar t-1}\gamma^{k}\norm{(\calP^\pi)^{k}g}_{2, \bar d_T^{\pi^b}}\\
		& \leq \sum_{k = 0}^{\bar t-1}\gamma^{k}\sqrt{\frac{p_{\max}}{p_{\min}}}\norm{g}_{2, \bar d_T^{\pi^b}}\\
		& \leq \frac{1-\gamma^{\bar t}}{1-\gamma}\sqrt{\frac{p_{\max}}{p_{\min}}}\norm{g}_{2, \bar d_T^{\pi^b}}.
		\end{align*}
		Summarizing together, we can obtain that
		\begin{align*}
		\norm{Q_1 - Q_2}_{2, \bar d_T^{\pi^b}} &\leq \frac{(1-\varepsilon)^{-1}(1-\gamma^{\bar t})}{1-\gamma}\sqrt{\frac{p_{\max}}{p_{\min}}}\norm{\calT h^\pi(Q_1 - Q_2)}_{2, \bar d_T^{\pi^b}},
		\end{align*}
		where we note that $\calT h^\pi(Q_1 - Q_2) = g$. Since $\varepsilon$ is arbitrary, let $\varepsilon$ go to $0$, we have
		$$
		\norm{Q_1 - Q_2}_{2, \bar d_T^{\pi^b}} \leq \frac{1}{1-\gamma}\sqrt{\frac{p_{\max}}{p_{\min}}}\norm{\calT h^\pi(Q_1 - Q_2)}_{2, \bar d_T^{\pi^b}}
		$$
		
		In the remaining proof, we show $\overline \tau$ is bounded above. Note that for any $Q \in L^2(S, A)$, 
		\begin{align*}
		\|h^\pi(Q)\|^2_{L^2(S, A, S')} &=  \overline\EE\left[\left(Q(S, A) - \gamma \int_{a' \in \calA} \pi(a' \given S')Q(S', a')\text{d}a'\right)^2\right]\\
		&\lesssim 2 \overline\EE\left[(Q(S, A))^2\right] + \frac{2p_{\max}\gamma^2}{p_{\min}}\int Q^2(s, a) \bar d^{\pi^b}_{T}(s, a)\text{d}s\text{d}a\\
		&\lesssim  (1 + \frac{p_{\max}\gamma^2}{p_{\min}}) \norm{Q}^2_{2, \bar d^{\pi^b}_T},
		\end{align*}
		where the first inequality is given by AM-GM, Jensen's inequalities and Assumption \ref{ass: L2 well-posed} by noting that $\bar d^{\pi^b}_{T+1}(s, a) \lesssim p_{\max}$ for any $s \in \calS$ and $a \in \calA$. Then by the first inequality given in \eqref{eq: L2 error bound}, we can show that
		$$
		\overline \tau \lesssim \frac{\sqrt{p_{\max}(1+\frac{p_{\max}\gamma^2}{p_{\min}})}}{(1-\gamma)\sqrt {p_{\min}}},
		$$
		which concludes our proof.
	\end{proof}
	
	Theorem \ref{thm: well-posed} 
	%is particularly useful and 
	rigorously justifies the validity of using $L^2$-norm to measure the Bellman error, which has been widely adopted in the existing literature for constructing various estimators for the $Q$-function. To see this, let $Q_1 = Q^\pi$ and $Q_2 = \widetilde Q$ in Theorem \ref{thm: well-posed}, where $\widetilde Q$ denotes some estimator for $Q^\pi$. Then the first inequality in \eqref{eq: L2 error bound} with Bellman equation \eqref{eq: Bellman equation for Q} implies that
	$$
	\norm{\widetilde Q- Q^\pi}_{2, \bar d_T^{\pi^b}} \lesssim \norm{r + (\gamma \calP^\pi - \calI)\widetilde Q  }_{2, \bar d_T^{\pi^b}},
	$$
	where the right hand side of the above inequality is called Bellman error (or residual) and recall that $r$ is the reward function defined in Assumption \ref{ass: reward}. Therefore $L^2$-norm of Bellman error of any $Q$-function estimator provides a valid upper bound for the $L^2$ error bound of this estimator to the true $Q^\pi$.  Many existing estimators  such as \cite{antos_learning_2008,farahmand2016regularized,uehara2019minimax,feng2020accountable} indeed are based on minimizing the $L^2$-norm of Bellman error. Therefore our Theorem \ref{thm: well-posed} provides a theoretical guarantee for their procedures. 
	Notice that Theorem \ref{thm: well-posed} is established without imposing any restriction on the structure of $Q$-function, it can be used to obtain $L^2$ error bounds for many different non-parametric estimators of $Q$-function obtained using different models and/or methods such as LSTD, kernel methods or neural networks. For example, combining our Theorem \ref{thm: well-posed} with Theorem 11 of \cite{farahmand2016regularized} immediately gives $L^2$-error bound for their estimator to the true $Q^\pi$. Applying our Theorem \ref{thm: well-posed} to Example 6 of \cite{uehara2021finite} one can obtain $L^2$-error bound for their neural network estimator to $Q^\pi$.  
	
	We also remark that the well-posed result in Theorem \ref{thm: well-posed} can be extended to other metrics such as $L^1$-norm, based on which one may develop a new estimator for $Q$-function by minimizing the empirical approximation of $L^1$-norm of Bellman error. We conjecture that such an estimator could achieve robustness compared with the existing ones, especially when the reward distribution is heavy tailed. Lastly, there is a very recent work \citep{wang2021projected}, which developed a sufficient and necessary condition for establishing the well-posedness of Bellman operator in $L^2$-norm. Besides they also developed a similar condition as our Assumption \ref{ass: L2 well-posed} in establishing this well-posedness.

	\section{Minimax Lower Bounds}\label{sec: lower bound}

	In this section, we establish minimax lower bounds in both sup-norm and in $L^2$-norm for estimation of nonparametric $Q$-function in OPE problem. The well-posedness property essentially indicates that non-parametric $Q$-function estimation is as easy as the classical non-parametric regression in the i.i.d. setting in terms of the worst case rate. 
	%See the minimax lower bound result in Theorem \ref{thm: L2 lower bound}. 
	
	Recall that by Theorem \ref{thm: well-posed}, under Assumptions \ref{ass: Markovian}, \ref{ass: DGP} and \ref{ass: L2 well-posed}, for any square integrable function $Q$ defined over $\calS \times \calA$, we have
	\begin{align}\label{eqn: well-posed operator}
	\sqrt{\frac{p_{\min}}{p_{\max}}}(1-\gamma)\norm{Q}_{2, \bar d^{\pi^b}_T} \leq \norm{\calT h^\pi(Q)}_{2, \bar d^{\pi^b}_T} \leq \norm{h^\pi(Q)}_{2, \bar d^{\pi^b}_T \times q} \lesssim \norm{Q}_{2, \bar d^{\pi^b}_T}.
	\end{align}
	Denote a generic transition tuple as $\{S_{i, t}, A_{i, t}, R_{i, t}, S'_{i, t}\}$ indexed by $(i, t)$. Then we have the following lower bound results for estimating $Q^\pi$ and its derivative in terms of the sup-norm.
	\begin{theorem}\label{thm: lower bound}
	Let $d^{\nu}$ be the average visitation probability density defined over $\calS \times \calA$ induced by some policy $\nu$ such that Assumption \ref{ass: L2 well-posed} holds with $\bar{d}^{\pi^b}_T$ and $\pi^b$ replaced by $d^\nu$ and $\nu$ respectively. Suppose the data $\calD_N =\{S_{i, t}, A_{i, t}, R_{i, t}, S'_{i, t}\}_{1 \leq i \leq N, 0 \leq t \leq T-1}$ are $i.i.d.$ from Model \eqref{Model: NPIV on $Q$-function}, where the probability density of $(S_{i, t}, A_{i, t})$ is $d^{\nu}$ with the transition probability density $q$ and for every $0 \leq t\leq T- 1$ and $1 \leq i \leq N$, $\EE[U_{i, t}^2 \given S_{i, t}, A_{i, t}] \geq \sigma^2$, where  $\sigma$ is some positive constant, then we have for any $0 \leq \norm{\alpha}_{\ell_1} < p$,
		\begin{align}
		\liminf_{NT \rightarrow \infty}\inf_{\wh Q} \sup_{Q \in \Lambda_\infty(p, L)}{\Pr}^Q\left(\norm{\partial^\alpha \wh Q - \partial^\alpha Q}_\infty \geq c (\log(NT)/NT)^{(p-\norm{\alpha}_{\ell_1})/(2p+d)}\right) \geq c' >0,
		\end{align}
		for some constants $c$ and $c'$, where $\inf_{\wh Q}$ denotes the infimum over all estimators using $\calD_N$, 
		and $\Pr^Q$ denotes the joint probability distribution of $\calD_N$ with $h^\pi = h^\pi(Q)$ in Model \eqref{Model: NPIV on $Q$-function}.
	\end{theorem}
	%As seen from Theorem \ref{thm: lower bound}, the proposed sieve 2SLS estimators for $Q^\pi$ and its derivatives achieve the minimax lower bounds, which shows their optimality. 
	The following theorem provides lower bound results in terms of $L^2$-norm. %also indicates the optimality of the results in Corollary \ref{coro: derivative L2}.
		\begin{theorem}\label{thm: L2 lower bound}
		Under all conditions in Theorem \ref{thm: lower bound}, for $0 \leq \norm{\alpha}_{\ell_1} < p$, we have
		\begin{align}
		\liminf_{NT \rightarrow \infty}\inf_{\wh Q} \sup_{Q \in \Lambda_2(p, L)}{\Pr}^Q\left(\norm{\partial^\alpha \wh Q - \partial^\alpha Q}_2 \geq \bar c (NT)^{(\norm{\alpha}_{\ell_1} - p)/(2p+d)}\right) \geq \bar c' >0,
		\end{align}
		for some constant $\bar c$ and $\bar c'$. %, where $\inf_{\wh Q}$ denotes the infimum over all estimators using $\{S_{i, t}, A_{i, t}, R_{i, t}, S'_{i, t}\}_{1 \leq i \leq N, 0 \leq t \leq T-1}$, and $\Pr^Q$ denotes the joint probability distribution of \linebreak $\{S_{i, t}, A_{i, t}, R_{i, t}, S'_{i, t}\}_{1 \leq i \leq N, 0 \leq t \leq T-1}$ with $h^\pi = h^\pi(Q)$ in Model \eqref{Model: NPIV on $Q$-function}.
	\end{theorem}
	As we can see from Theorems \ref{thm: lower bound} and \ref{thm: L2 lower bound}, the minimax lower bounds for the rates of estimating $Q$-function and its derivatives are the same as those for nonparametric regression in the i.i.d setting \citep{stone1982optimal}. To the best of our knowledge, these are the first lower bound results for nonparametrically estimating $Q$-function and its derivatives in the infinite-horizon MDP. In the following section, we proposed simple estimators that match these lower bounds.

	\section{Sieve 2SLS Estimation of $Q$-function}\label{sec: NPIV}

	%In this section, we show that the sup-norm rates obtained from Theorem \ref{thm: sup-norm rate} and Corollary \ref{coro: %derivative sup-norm} are minimax-optimal.
	
	Given the NPIV Model \eqref{Model: NPIV on $Q$-function} as a reformulation of Bellman equation, we now adopt the idea from for example \cite{blundell2007semi,chen2013optimal} to construct a sieve 2SLS estimator for $Q^\pi$. Define two sieve basis functions as
	\begin{align}\label{eq: sieve basis function}
	&\psi^J(s, a) = (\psi_{J1}(s, a), \cdots, \psi_{JJ}(s, a))^\top,\\
	&b^K(s, a) = (b_{K1}(s, a), \cdots, b_{KK}(s, a))^\top,
	\end{align}
	to model $Q^\pi$ and  the space of instrumental variables respectively. Let $\Psi_J = \text{closure}\{\Psi_{J1},\ldots,\Psi_{JJ}\} \subset L^2(S, A)$ and $B_K = \text{closure} \{b_{K1},\ldots,b_{KK}\}\subset L^2(S, A)$ denote the sieve spaces for $Q^\pi$ and instrumental variables, respectively. Here the underlying probability measure of $L^2(S, A)$ is $\bar d^{\pi^b}_T$. Examples of basis functions include splines or wavelet bases (See more examples in \cite{huang1998projection,chen2007large}). The construction of wavelet bases can also be found in Appendix \ref{sec: lower bound}. We remark that the numbers of basis functions $J$ and $K$ are allowed to grow with either $N$ or $T$, but require that $J \leq K \leq cJ$ for some $c \geq 1$. Due to the special structure of Model \eqref{Model: NPIV on $Q$-function}, it also makes sense to simply let $K=J$ and $\psi^J =b^K$. %In the Section \ref{subsec: L2 norm rates}, we illustrate the benefits of using different basis functions.
	
	Additionally, we let $\psi^J_\pi(s) = (\int_{a\in \calA}\pi(a | s)\psi_{J1}(s, a)\text{d}a, \cdots, \int_{a\in \calA}\pi(a | s)\psi_{JJ}(s, a)\text{d}a)^\top$. Correspondingly, a sample version of all these functions can be defined as
	\begin{align*}
	&\Psi = \left(\psi^J(S_{1, 0}, A_{1, 0}) \cdots, \psi^J(S_{N, T-1}, A_{N, T-1})\right)^\top \in \mathbb{R}^{(NT) \times J},\\
	&B = (b^K(S_{1, 0}, A_{1, 0}),b^K(S_{1, 1}, A_{1, 1}) \cdots, b^K(S_{N, T-1}, b_{N, T-1}))^\top \in \mathbb{R}^{(NT) \times K},\\
	& G_\pi = (\psi^J_\pi(S_{1, 1}),\psi^J_\pi(S_{1, 2}) \cdots,\psi^J_\pi(S_{1, T}), \psi^J_\pi(S_{2, 1}), \cdots,  \psi^J_\pi(S_{N, T}))^\top \in \mathbb{R}^{(NT) \times J}.
	\end{align*}
	For notational simplicity, let
	$
	\kappa_\pi^J(s, a, s') = \psi^J(s, a) - \gamma \psi_\pi^J(s')$, and correspondingly $ \Gamma_\pi = \Psi - \gamma G_\pi.
	$
	We also denote $\kappa_{Jj}^\pi(s, a, s') = \psi_{Jj}(s, a) - \gamma \int_{a' \in \calA}\pi(a' \mid s')\text{d}a' \psi^\pi_{Jj}(s', a')$ for each element of $\kappa_\pi^J(s, a, s') \in \mathbb{R}^J$.
	Then the sieve 2SLS estimator for $Q^\pi$ can be constructed as
	\begin{align}\label{eqn: sieve estimator}
	&\wh Q^\pi(s, a) = \psi^J(s, a)^\top \wh c \nonumber, \\
	& \text{with}\quad  \wh c  = \left[\Gamma_\pi^\top B(B^\top B)^{-}B^\top \Gamma_\pi\right]^{-}\Gamma_\pi^\top B(B^\top B)^{-}B^\top \mathbf R,
	\end{align}
	where $(Z)^{-}$ denotes the generalized inverse of some matrix $Z$ and $\mathbf R = (R_{1, 0},R_{1, 1} \cdots, R_{N, T-1})^\top \in \mathbb{R}^{NT \times 1}$. The corresponding estimator for the derivatives of $Q^\pi$ is denoted by $\partial^\alpha \wh Q^\pi$ for any vector $\alpha$. Here $\wh c$ can be understood as a minimizer of the following optimization problem.
	\begin{align*}
	\underset{c \in \mathbb{R}^{J}}{minimize} \quad \norm{B(B^\top B)^{-1}B^\top(\mathbf R - \Gamma_\pi c)}^2_{\ell_2}.
	\end{align*}
	Note that the sieve 2SLS  estimator given in \eqref{eqn: sieve estimator} becomes the solution of the modified Bellman residual minimiazion in \cite{farahmand2016regularized} when their function spaces are modeled by sieve ones. 
	
	\subsection{Sieve measure of ill-posedness in NPIV}
	
	An important quantity related to a generic NPIV model \eqref{Model:NPIV} is called \textit{sieve $L^2$ measure of ill-posedness}, which characterizes the difficulty of non-parametrically estimating $h_0$ using the sieve estimation. Here a similar measure of ill-posedness can be defined under Model \eqref{Model: NPIV on $Q$-function}. 
	Let $\Theta^\pi_{J} = \{h^\pi(Q) \in L^2(S, A, S'): Q \in \Psi_J\}$. Adapting from the {sieve $L^2$ measure of ill-posedness} in \cite{blundell2007semi}, we define an average \textit{sieve $L^2$ measure of ill-posedness} across $T$ decision points  under Model \eqref{Model: NPIV on $Q$-function} as
	\begin{equation}\label{def: sieve measure of ill-posedness}
	\tau_J  =  \sup_{h \in \Theta^\pi_J : h \neq 0} \frac{\|h\|_{L^2(S, A, S')}}{\|\calT h\|_{L^2(S, A)}}\,. %\quad \text{with density of $(S, A, S')$ as $\bar d^{\pi^b}_T \times q$}.
	\end{equation}
	It can be seen that $\tau_J \geq 1$.
	 Basically $\tau_J $ measures how much information has been smoothed out by the conditional expectation operator $\calT$ over the space $\Theta^\pi_J$. For a generic NPIV model \eqref{Model:NPIV}, $\tau_J$ grows to infinity as $J$ goes to infinity; see, e.g. \citep{blundell2007semi,chen2013optimal}.
	 %the mildly and severely ill-posed cases were analyzed, i.e., $\tau_J$ grows polynomially or exponentially along with $J$ %respectively. 
	% In the following, we show that, under some mild assumptions, the defined $\tau_J$ under Model \eqref{Model: NPIV on %$Q$-function} for OPE problem is uniformly bounded above by some constant for all $J$ This shows that the special form of NPIV %Model \eqref{Model: NPIV on $Q$-function} is well-posed in terms of sieve measure of ill-posedness as well.
By definition we have $ \tau_J \leq \overline{\tau} \lesssim 1$ for all $J \geq 1$. Thus Theorem \ref{thm: well-posed} directly implies that the NPIV Model \eqref{Model: NPIV on $Q$-function} is also well-posed under the $L^2$ sieve measure of ill-posedness defined in \eqref{def: sieve measure of ill-posedness}.
Based on this result,  minimax-optimal sup-norm and $L^2$-norm rates for the sieve 2SLS estimator of $Q^\pi$ can be established in the following subsections.

	\subsection{ Sup-norm Convergence Rates}\label{sec: sup-norm rate}
	In this subsection, we establish the sup-norm convergence rate of $\wh Q^\pi$ to $Q^\pi$. We first introduce an additional assumption on the data generating process.
	\begin{assumption}\label{ass: stationary}
	    The stochastic process $\{S_t, A_t\}_{t\geq 0}$ induced by the behavior policy $\pi^b$ is a stationary, exponentially $\boldsymbol{\beta}$-mixing stochastic process, i.e., the $\boldsymbol{\beta}$-mixing coefficient at time lag $k$ satisfies that $\beta_k \leq \beta_0 \exp(-\beta_1 k)$ for $\beta_0 \geq 0$ and $\beta_1 > 0$. The induced stationary density is denoted by $d^{\pi^b}$.
		%The behavior policy is stationary depending on the current state, denoted by $\pi^b$.
	\end{assumption}
	Assumption \ref{ass: stationary}  is imposed to characterize the dependency among observations over time because the observed data modeled by MDP are \textit{not} i.i.d. and transition tuples are dependent. Most of previous works assume transition tuples are independent, which is stronger than this assumption. The  $\boldsymbol{\beta}$-mixing coefficient at time lag $k$ basically means that the dependency between $\{S_t, A_t\}_{t \leq j }$	and $\{S_t, A_t\}_{t \geq (j+k) }$ decays to 0 at an exponential rate with respect to $k$. See \cite{bradley2005basic} for the exact definition of the exponentially $\boldsymbol{\beta}$-mixing. A fast mixing rate is imposed here mainly for technical simplicity and our sup-norm and $L^2$-norm convergence rates are not affected by the mixing coefficients. Indeed, Assumption \ref{ass: stationary} can be relaxed to stationary with certain algebraic $\beta$-mixing, and by using the matrix Bernstein inequality for general $\beta$-mixing of \cite{chen2015optimal} one may obtain the same convergence rates as those in Theorems \ref{thm: sup-norm rate} and \ref{thm: L2 variance+bias bound} below. We can also relax the strictly stationary assumption with some extra notation. Since this is not our focus, we do not impose the weakest possible assumptions on the temporal dependence in this paper. When both Assumptions \ref{ass: L2 well-posed}  and \ref{ass: stationary} hold, the average visitation probability density $\bar d^{\pi^b}_T$ used in Assumption \ref{ass: L2 well-posed} becomes the induced stationary density $d^{\pi^b}$. We will omit $\nu$ in $\norm{\bullet}_{2, \nu}$ when $\nu = d^{\pi^b}$.  Throughout the remaining of this section, unless otherwise specified,  $(S, A)$ has the probability density $d^{\pi^b}$ and the density of $(S, A, S')$ is $d^{\pi^b} \times q$.
	%\begin{assumption}\label{ass: basis functions}
	%	The sieve space $\Psi_J$ is spanned by a wavelet basis of \citep{cohen1993wavelets} or B-spline. $B_K$ are spanned by a wavelet, B-spline basis or cosine basis. 
	%\end{assumption}
	%For a detailed review of these basis functions, we refer to Appendix E of \citep{chen2013optimal}. 
	
	Define $L_{2,h^\pi}(S, A, S') = \left\{h^\pi(Q): Q \in L^2(S, A) \right\}$.
	Let  $\Pi_J: L_{2, h^\pi}(S, A, S') \to \Theta^\pi_J$ denote the $L_{h^\pi}^2(S, A, S')$ mapping onto $\Theta^\pi_J$, i.e., $\Pi_J h_0^\pi = h_0^\pi(\overline \Pi_J Q^\pi)$, where $\overline \Pi_J Q^\pi = \mathrm{arg}\min_{Q \in \Psi_J} \|Q^\pi - Q\|_{L^2(S, A)}$, and let $\Pi_K :L^2(S, A) \to B_K$ denote the $L^2(S, A)$ orthogonal projection onto $B_K$.
	Let $\widetilde \Pi_J h_0^\pi = \mathrm{arg}\min_{h \in \Theta^\pi_J} \|\Pi_K \calT(h_0^\pi - h)\|_{L^2(S, A)}$ denote the sieve 2SLS projection of $h_0^\pi$ onto $\Theta^\pi_J$. Let $\Theta_{J, 1}^\pi = \{h \in \Theta_{J}^\pi \given \norm{h}_{2} = 1  \}$.
	%We may write $\widetilde \Pi_J h_0^\pi = \kappa^\pi_J(\bullet, \bullet, \bullet)\top c_{0,J}$ where
	%\[
	%c_{0,J} = [\Sigma^{\pi \, \top} G_b^{-1} \Sigma^\pi]^{-1} \Sigma^{\pi \, \top}  G_b^{-1} \EE[b^K(S, A)h_0^\pi(S, A, S')]\,.
	%\]
	We make one additional assumption below for controlling the approximation error of using the sieve bases.
	\begin{assumption}\label{ass: approx}
		(a) $\sup_{h \in \Theta^\pi_{J, 1}}\|(\Pi_K \calT - \calT)h\|_{L^2(S, A)} = o_J(1)$, where $o_J(1)$ refers to a quantity that converges to $0$ when $J \rightarrow \infty$; (b) $\|\widetilde{\Pi}_J (h^\pi_0 - \Pi_J h^\pi_0) \|_\infty \leq C_1\times \|h^\pi_0 - \Pi_J h^\pi_0\|_{\infty}$ for some constants $C_1$.
	\end{assumption}
	Assumption \ref{ass: approx}~(a) is a mild condition on approximating $\Theta^\pi_J$ by a sieve space $B_K$. For fixed $J$ (and $K$), $\sup_{h \in \Theta^\pi_{J, 1}}\|(\Pi_K \calT - \calT)h\|_{L^2(S, A)}$ can be interpreted as an inherent Bellman error (for a fixed policy $\pi$), which is widely used in the literature of RL such as the analysis of fitted-q iteration (See Assumption 4.2 of \cite{agarwal2019reinforcement}). %Here we do not require a specific convergence rate for the inherent Bellman error towards zero. %Assumption \ref{ass: approx}~(b) is a \textit{$L_\infty$ stability condition} used in NPIV literature such as \citep{chen2013optimal} to derive the sup-norm rate. 
	Assumption \ref{ass: approx}~(b) is also mild because that $\norm{\widetilde \Pi_J (h_0^\pi - \Pi_J h_0^\pi)}_{L^2(S, A)} \leq \norm{ h_0^\pi - \Pi_J h_0^\pi}_{L^2(S, A)}$ holds automatically by the projection property. Here we strengthen it in terms of the sup-norm.
	
	To derive the sup-norm convergence rate, following the proof of \cite{chen2013optimal}, we split $\|\wh Q^\pi - Q^\pi\|_\infty$ into two terms. Let $
	\widetilde Q^\pi(s, a)  = \psi^J(s, a)^\top \widetilde c ~~\mbox{ with }~~ \widetilde c = [\Gamma_\pi^\top B(B^\top B)^-B^\top\Gamma_\pi]^- \Gamma_\pi^ \top B (B^\top B)^- B^\top H_0,
	$
	where 
	$$
	H_0 = (h_0^\pi(S_{1, 0}, A_{1, 0}, S_{1, 1}),h_0^\pi(S_{1, 1}, A_{1, 1}, S_{1, 2}), \ldots, h_0^\pi(S_{N, T-1}, A_{N, T-1}, S_{N, T}))^\top \in \mathbb{R}^{NT}.
	$$ 
	Then by triangle inequality, we have $\|\wh Q^\pi - Q^\pi\|_\infty \leq \|\wh Q^\pi - \widetilde Q^\pi\|_\infty + \|Q^\pi - \widetilde Q^\pi\|_\infty$. The first term $\|\wh Q^\pi - \widetilde Q^\pi\|_\infty$ can be interpreted as an estimation error, while the second term $\|Q^\pi - \widetilde Q^\pi\|_\infty$ can be understood as the approximation error. Denote $ G^\pi_{\kappa,J} = \EE[\kappa^J_\pi(S, A, S')\kappa^J_\pi(S, A, S')^\top]$ and $e_J = \lambda_{\min}(G^\pi_{\kappa,J})$.
	Let
	\begin{align*}
	\zeta^\pi_{\kappa} & =  \sup_{s, a, s'} \|[G^\pi_{\kappa, J}]^{-1/2} \kappa^J_\pi(s, a, s')\|_{\ell^2} & & G_b = \EE[b^K(S, A)b^K(S, A)^\top] \\
	\xi_{\psi} & = \sup_{s, a} \|\psi^J (s, a)\|_{\ell^1} && \zeta_{b}    =  \sup_{s, a} \|G_b^{-1/2} b^K(s, a)\|_{\ell^2}
	\end{align*}
	for each $J$ and $K$ by omitting their dependence on $J$ and $K$ for notation simplicity, and define $\zeta = \max\{\zeta_{b}, \zeta^\pi_{\kappa}\}$.  In the following lemma, we derive bounds for the aforementioned two terms.
	\begin{lemma}\label{lm: variance+bias bound}
	    (1) Let Assumptions \ref{ass: Markovian}-\ref{ass: stationary}, and Assumption \ref{ass: approx}(a) hold. If $ \zeta^2\sqrt{\log (NT)/NT}= O(1)$, then we have the following sup-norm bound for the estimation error.
		\begin{align}\label{eqn: std error}
		\|\wh Q^\pi - \widetilde Q^\pi\|_\infty = O_p \Big(\xi_{\psi} \sqrt{(\log J)/(NT e_J)}\Big).
		\end{align}
		(2) Let Assumptions \ref{ass: stationary}-\ref{ass: approx} hold. If $\zeta^2\sqrt{\log(J)\log (NT)/NT}= O(1)$ then the approximation error can be controlled by
		\begin{align}\label{eqn: approximation error bound}
		    \|Q^\pi - \widetilde Q^\pi\|_\infty = O_p(\|Q^\pi- \overline \Pi_JQ^\pi \|_\infty).
		\end{align}
	\end{lemma}
    By examining the proof of Lemma \ref{lm: variance+bias bound}, it is possible to derive the finite sample error bounds for 
    both $\|\wh Q^\pi - \widetilde Q^\pi\|_\infty$ and  $\|Q^\pi - \widetilde Q^\pi\|_\infty$. We omit them for brevity. 
	\begin{theorem}\label{thm: sup-norm rate}
		 Let Assumptions \ref{ass: Markovian}-\ref{ass: approx} hold and $Q^\pi \in \Lambda_\infty(p,L)$ for some $L$. Suppose that the sieve space $\Psi_J$ is spanned by a B-spline or wavelet basis of \cite{cohen1993wavelets} with regularity larger than $p$, and $B_K$ is spanned by a wavelet, spline or cosine basis. If $J\sqrt{\log(J)\log (NT)/NT} = O(1)$, then we have:
		\begin{equation}\label{eqn: h error}
		\| \wh Q^\pi - Q^\pi \|_{\infty} = O_p \big(J^{-p/d} +  \sqrt{J(\log J)/(NT)}\big).
		\end{equation}
		Further, by choosing $ J \asymp (\frac{NT}{\log(NT)})^{d/(2p+d)}$ and assuming $2p > d$ , we have for all $0\leq \norm{\alpha}_{\ell_1} < p$,
		\begin{equation}\label{eqn: derivatives Q error}
		\| \partial^\alpha \wh Q^\pi -  \partial^\alpha Q^\pi \|_{\infty} = O_p\left(\left(\frac{\log(NT)}{NT} \right)^{(p - \norm{\alpha}_{\ell_1})/(2p+d)}  \right).
		\end{equation}
	%	\begin{equation}\label{eqn: Q error}
	%	\| \wh Q^\pi -  Q^\pi \|_{\infty} = O_p\left(\left(\frac{\log(NT)}{NT} \right)^{p/(2p+d)}  \right).
	%	\end{equation}
		%	Furthermore, if an additional assumption that $\tau_J\norm{\calT (h_0^\pi - \Pi_J h_0^\pi)}_{L_2(S, A)} \leq C_2\norm{ h_0^\pi - \Pi_J h_0^\pi}_{L_2(S, A)}$ is satisfied for some universal constant $C_2$, then
		%	\begin{equation}\label{eqn: Q error 2}
		%	\| \wh Q^\pi -  Q^\pi \|_{\infty} = O_p\left(\frac{R_{\max}\sqrt{(1+\frac{p_{\max}\gamma^2}{p_{\min}})}}{(1-\gamma)^2}\left(\frac{\log(NT)}{NT} \right)^{p/(2p+d)} \right).
		%	\end{equation}
	\end{theorem}
	The smoothness parameter $p$ in Theorem \ref{thm: sup-norm rate} represents the smoothness level of the true $Q$-function and characterize the size of the functional class that the true $Q$-function belongs to. We require $2p > d$ in Theorem \ref{thm: sup-norm rate} mainly for deriving the sup-norm rate in order to achieve the optimality when considering H\"older class of functions.  Theorem \ref{thm: sup-norm rate} shows that in terms of the batch data sample size of $NT$, the sup-norm rate of our 2SLS estimator $\widehat{Q}^\pi$ to $Q^\pi$ is the same as the optimal one in the classical non-parametric regression estimation \citep{stone1982optimal} using B-splines. %The first term in \eqref{eqn: h error} represents the upper bound for the approximation error, while the second term characterizes the standard deviation of the stochastic error. 
 The sup-norm convergence rate of $\wh Q^\pi$ can be useful to develop uniform confidence bands (UCBs) for the $Q$-function using results in \cite{chen2013optimal} for example. Such UCBs may be incorporated into the framework of pessimistic RL algorithms such as \citep{jin2021pessimism,xie2021bellman}. In addition, we can also show that the sup-norm bounds on the constant factor $\gamma$ is of order $(1-\gamma)^{-3}$.
	%Since we consider the continuous action space $\calA$, it may be interesting to estimate the derivatives of $Q$-function. 
	%As we discussed after Corollary \ref{coro: sieve well-posed}, if we consider two ill-posed cases, for the mildly ill-posed one, the sup-norm rate of $\widehat{Q}^\pi$ becomes $O_p((\frac{\log(NT)}{NT} )^{p/(2(p+2\varsigma+d))}$, and becomes $O_p((\log(NT) )^{-p/(2\varsigma)}$ for the severely ill-posed one. The corresponding $J \asymp (NT/\log(NT)^{d/(2(p+2\varsigma)+d)}$ and $J \asymp (c_0\log(NT))^{d/2\varsigma}$ with $c_0 \in (0,1)$ respectively.
%	Next we provide sup-norm rates on the proposed estimator $\partial^\alpha \wh Q^\pi$ to the derivatives of $\partial^\alpha %Q^\pi$.
	%\begin{corollary}\label{coro: derivative sup-norm}
	%	Let conditions in Theorem \ref{thm: sup-norm rate} hold. If $ J \asymp (\frac{NT}{\log(NT)})^{d/(2p+d)}$, then: for all $0 < %\norm{\alpha}_{\ell_1} < p$,
	%	\begin{equation}\label{eqn: derivatives Q error}
	%	\| \partial^\alpha \wh Q^\pi -  \partial^\alpha Q^\pi \|_{\infty} = O_p\left(\left(\frac{\log(NT)}{NT} \right)^{(p - %\norm{\alpha}_{\ell_1})/(2p+d)}  \right).
	%	\end{equation}
	%	\end{corollary}
%	Corollary \ref{coro: derivative sup-norm} extends the result in Theorem \ref{thm: sup-norm rate} if the interest becomes 
% estimating the derivatives of $Q$-function. 
	Finally, the results on the sup-norm rates for estimating the derivatives of the $Q$-function may be useful to some actor-critic algorithms such as \cite{silver2014deterministic,kallus2020statistically,xu2021doubly}.
	
	\subsection{$L^2$-norm Convergence Rates}\label{subsec: L2 norm rates}
	In this subsection we present the $L^2$ convergence rates of our 2SLS estimator for the $Q$-function. We do not require Assumption \ref{ass: approx}~(b) as the $L^2$-stability condition holds automatically.
	\begin{theorem}\label{thm: L2 variance+bias bound}
	    Let Assumptions \ref{ass: Markovian}-\ref{ass: approx}~(a) hold. If $\zeta\sqrt{\log (NT)\log(J)/NT} = o(1)$, then:
		\begin{align}\label{eqn: L2 std error+ approximation}
		\|\wh Q^\pi -  Q^\pi\|_2 = O_p \Big( \sqrt{J/(NT)} + \norm{Q^\pi - \overline \Pi_J Q^\pi}_2\Big).
		\end{align}
	If $Q^\pi \in \Lambda_2(p,L)$ with $p>0$, and $\Psi_J$ and $B_K$ are spanned by some commonly used bases such as polynomials, trigonometric polynomials, splines and wavelets with regularity greater than $p$, by choosing $J \asymp (NT)^{d/(2p+d)}$, we have: for all $0\leq \norm{\alpha}_{\ell_1} <p$,
	\begin{equation*}\label{eqn: L2 Q error}
		\| \partial^\alpha \wh Q^\pi -  \partial^\alpha Q^\pi \|_{2} = O_p\left(\left(NT\right)^{( \norm{\alpha}_{\ell_1} - p)/(2p+d)}  \right).
	\end{equation*}
	\end{theorem}
	According to Theorem \ref{thm: L2 variance+bias bound}, the sieve 2SLS estimator $\wh Q^\pi$ achieves the minimax optimal $L^2$-norm convergence rate to $Q^\pi$ under conditions much weaker than those for the optimal sup-norm convergence rate. Let $F$ be a known marginal distribution. It is well-known that one can estimate the value of a target policy $\pi$
	   \[
	   v^{\pi}=\int_{s \in \calS}[\int_{a \in \calA} \pi(a \given S=s)Q^\pi(s, a)\text{d}a]F(\text{d}s)
	   \]
	  by the following simple plug-in sieve 2SLS estimator
	  \[
	  \wh v^{\pi}=\int_{s \in \calS}[\int_{a \in \calA} \pi(a \given S=s) \wh Q^\pi(s, a)\text{d}a]F(\text{d}s).
	  \]
	Theorem \ref{thm: L2 variance+bias bound} is particularly useful in establishing the asymptotic normality of $\sqrt{NT} (\wh v^{\pi}- v^{\pi})$.
	% By applying asymptotic efficiency result of \cite{ai2003efficient} for  and efficiency of  $\wh v^{\pi}$ for the value  
	% $v^{\pi}$ of the target policy.
	
	\begin{remark}\label{rm:Shi}
	(1) Recently \cite{shi2020statistical} presented a sieve LSTD estimator for $Q^\pi$ and obtained the $L^2$-norm rate of convergence (See (E.46) in appendix of their paper for more details) under some conditions including their Assumption (A3.) or a small discount factor $\gamma$ condition. They then apply their $L^2$-norm convergence rate to establish the $\sqrt{NT}$-asymptotic normality of plug-in sieve LSTD estimator for the value $v^{\pi}$. Note that their sieve LSTD is a special case of our sieve 2SLS with $B_K=\Psi_J$ and $K=J$, and the sieve LSTD automatically satisfies our Assumption \ref{ass: approx}~(a). Our Theorem \ref{thm: L2 variance+bias bound} establishes the $L^2$-norm convergene rate for their sieve LSTD estimator without the need to impose the strong condition of a small discount factor $\gamma$. Thus we may require weaker conditions for establishing the asymptotic normality of the plug-in sieve 2SLS estimator for the value. We leave details to the longer version of the paper.
	 %, demonstrating the superiority of our theoretical results. 
	 (2) In this paper, to obtain the optimal rates of convergence in $L^2$-norm (and sup-norm) of our sieve 2SLS estimator for the $Q^\pi$ function, we assume strictly stationary data for simplicity. We note that \cite{shi2020statistical} did not impose this strict stationarity in their $L^2$-norm rate and asymptotic normality calculation. However, they need to assume the distribution of the initial state $S_{i,0}$ in the batch data is bounded away from $0$ uniformly in $i$. Indeed it is possible to replace the strict stationary condition in our Assumption \ref{ass: stationary} by imposing the geometric ergodicity and using the truncation argument to obtain the same sup-norm and $L^2$-norm convergence rates for our sieve 2SLS estimator. We leave it for the future work. 
	\end{remark}

	\section{Conclusion}\label{sec: conclusion}
	In this paper, we consider nonparametric estimation of $Q$-function of continuous states and actions in the OPE setting.
	Under some mild conditions, we show that the NPIV model \eqref{Model: NPIV on $Q$-function} for estimating $Q$-function nonparametrically is well-posed in the sense of $L^2$-measure of ill-posedness, bypassing the need of imposing a strong condition on the discount factor $\gamma$ in the recent literature. The well-posedness property effectively implies that the minimax lower bounds for nonparametric estimation of $Q$-function coincide with those for a nonparametric regression in sup-norm and in $L^2$-norm under the i.i.d. setting. Under mild sufficient conditions, we also establish that the sup-norm and the $L^2$-norm rates of convergence of our proposed sieve 2SLS estimators for $Q$-function achieve the lower bounds, and hence are minimax-optimal. These rate results are useful for optimal estimation and inference on various functionals, such as the value, of the $Q$-function by plugging in our sieve 2SLS estimators. In particular, one can easily develop uniform confidence bands (UCBs) for the $Q$-function by slightly modifying the UCBs result in \cite{chen2013optimal} for a NPIV function estimated via a spline or wavelet sieve 2SLS. We leave this to future work due to the length of the paper.
	%adapt the uniform confidence bands result in  for NPIV models to one extension of our work, it may be interesting to develop %uniform confidence bands for $Q$-function based on the sup-norm rates by leveraging the results in \citep{chen2013optimal} for %uniform confidence bands for linear and nonlinear functionals of $Q$-function estimated via plug-in sieve 2SLS.  %\QZL{To add %more.} %Such results could be beneficial to investigating other parameters in OPE. 
	
	In this paper we focus on the direct method of using Bellman equation to nonparametrically estimate $Q$-function in the OPE setting. 
	%But our new results of well-posedness are also useful to other approaches in OPE settings. 
	In the existing literature, there are two additional approaches to perform OPE. One is using the recently proposed marginal importance sampling for the infinite horizon setting  such as \cite{liu2018breaking, nachum2019dualdice,xie2019towards,uehara2020minimax,zhang2020gendice,zhang2020gradientdice}. The other approach combines the direct method and marginal importance sampling to construct the so-called doubly robust estimators for the value of the target policy (see, e.g., \cite{kallus2019efficiently, tang2019doubly,shi2021deeply} among many others). Our results on the well-posedness and the minimax lower bounds for $Q$ function estimation should be useful to establish theoretical properties of these alternative approaches under conditions that are weaker than the existing ones.
	%or the recently proposed data-adaptive approach in \citep{chen2021adaptive}.
	Finally, since OPE serves as the foundation of many RL algorithms, our results on $Q$-function estimation of a target policy can also be useful to other policy learning methods such as those proposed in  \cite{ernst_tree-based_2005,antos_learning_2008,le2019batch,liao2020batch,jin2021pessimism,zanette2021provable}.  We leave details to future work.

	\newpage
	\appendix
	\section{Notation}
	Unless specified, for any transition tuple $(S, A, S')$, the probability density of $(S, A)$ is $\sim d^{\pi_b}$ and the probability density of $S'$ given $(S, A)$ is $q$. In addition, $\EE$ refers to the expectation taken with respect to $d^{\pi_b}$. We recall the definition of some quantities below, which will appear in our proof.
	\[
	\begin{array}{rcccl}
	G^\pi_\kappa & = & G^\pi_{\kappa,J} & = & \EE[\kappa^J_\pi(S, A, S')\kappa^J_\pi(S, A, S')^\top]  =  \EE[ \Gamma_\pi^\top\Gamma_\pi/(NT)]\\
	G_b & = & G_{b,K} & = & \EE[b^K(S, A)b^K(S, A)^\top]  =  \EE[ B^\top B/(NT)]\\
	G_\psi & = & G_{\psi,J} & = & \EE[\psi^J(S, A)\psi^J(S, A)^\top]\\
	\Sigma^\pi & = & \Sigma^\pi_{K, J} & = & \EE[b^K(S, A)\kappa^J_\pi(S, A, S')^\top]  =  \EE[ B^\top\Gamma_\pi/(NT)]\,.
	\end{array}
	\]
	We assume that $\Sigma^\pi$ has a full column rank $J$. Denote $e_J = \lambda_{\min}(G^\pi_{\kappa,J})$. %and $e_{b,K} = \lambda_{\min}(G_{b,K})$.  
	Let
	\begin{align*}
	\zeta^\pi_{\kappa} & = \zeta^\pi_{\kappa, J}  =  \sup_{s, a, s'} \|[G^\pi_{\kappa, J}]^{-1/2} \kappa^J_\pi(s, a, s')\|_{\ell^2} & & \zeta_{b}  = \zeta_{b,K}  =  \sup_{s, a} \|G_b^{-1/2} b^K(s, a)\|_{\ell^2}\\
	\xi^\pi_{\kappa} & = \xi^\pi_{\kappa,J} = \sup_{s, a, s'} \|\kappa^J_\pi (s, a, s')\|_{\ell^1} & & \xi_{\psi} = \xi_{J} = \sup_{s, a} \|\psi^J_\pi (s, a)\|_{\ell^1}
	\end{align*}
	for each $J$ and $K$, and $\zeta= \max\{\zeta_{b,K}, \zeta^\pi_{\kappa,J}\}$. Define 
		$$
		(G_b^{-1/2} \Sigma^\pi)^-_l = \left[[\Sigma^{\pi}]^\top (G_b)^{-1}\Sigma^\pi\right]^{-1}[\Sigma^{\pi}]^\top (G_b)^{-1/2},
		$$
		and similarly for $(\wh G_b^{-1/2} \wh \Sigma^\pi)^-_l$, where
	$$
	\wh \Sigma^\pi = \frac{B^\top \Gamma_\pi}{NT} \quad \text{and} \quad \wh G_b = \frac{B^\top B}{NT}.
	$$

	\section{Lower bounds}\label{sec: lower bound}
	In this section, the probability density of $(S_{i, t}, A_{i, t})$ is $d^{\nu}$ and the expectation is with respect to the density $d^{\nu}$.
	\subsection{Lower bounds for Sup-norm Rates}
	The proof mainly follows that of Theorem 3.2 in \cite{chen2013optimal}. Consider the Gaussian reduced form of NPIV model with known operator $\calT$:
	\begin{align}\label{Model: gaussian model}
	R_{i, t} & = \calT h^\pi_0(S_{i, t}, A_{i, t}) + U_{i, t},\\
	U_{i, t} \mid (S_{i, t}, A_{i, t}) &\sim \calN(0, \sigma^2(S_{i, t}, A_{i, t}))\nonumber,
	\end{align}
	for $1 \leq i \leq N$ and $0 \leq t \leq T-1$.  The known of operator $\calT$ is equivalent to knowing the transition density $q$. By Lemma 1 of \cite{chen2011rate}, the minimax lower bound of Model \eqref{Model: NPIV on $Q$-function} is no smaller than Model \eqref{Model: gaussian model}. In the following, we thus focus on Model \eqref{Model: gaussian model} and make use of Theorem 2.5 of \cite{Tsybakov2009}.
	
	We restrict $\calS \times \calA = [0, 1]^d$.  Let $\{\widetilde \phi_{ j, k, G}, \widetilde\psi_{j, k, G} \}_{j, k, G}$ be a tensor-product wavelet basis of regularity larger than $p$ for $L^2([0, 1]^d)$, where $j$ is the resolution level, $k = (k_1, k_2, \cdots, k_d) \in \{0, 1, \cdots, 2^{j} - 1\}^d$, and $G$ is a vector indicating which element in a Daubechies pair $\{\phi, \psi \}$ is used. Note that $\phi$ has support $[-M+1, M]$ for some positive integer $M$. All these pairs are generated by CDV wavelets \citep{cohen1993wavelets}. Following the proof of \cite{chen2013optimal}, we consider a class of submodels around $Q^\pi$. In particular, for a given $j$, consider the wavelet space $(\calS \times \calA)_j$, which consists of $2^{jd}$ functions $\{\widetilde \psi_{j, k, G}\}_{k \in \{0, \cdots, 2^j - 1\}^d }$ with $G$ chosen as all $\psi$ functions. For some constant $r$, consider $\{\widetilde \psi_{j, k, G}\}_{k \in \{r, \cdots,  2^j - 1-M\}^d}$ as interior wavelets and $\widetilde \psi_{j, k, G}(s, a) = \Pi_{m=1}^{d-1} \psi_{j, k_m}(s_m)\psi_{j, k_d}(a)$ for $k = (k_1, \cdots, k_d) \in \{r, \cdots,  2^j - 1-M\}^d$, where $s = (s_1, \cdots, s_{d-1}) \in \calS$, $a \in \calA$ and $\psi_{j, k_m}(\bullet) = 2^{j/2}\psi(2^j(\bullet) - k_m)$ for $1 \leq m \leq d$.
	Then for sufficiently large $j$, there exists a set $\calI \subseteq \{r, \cdots, (2^j - M - 1)\}^d$ of interior wavelets with $\text{Card}(\calI) \gtrsim 2^{dj}$, where $\text{Card}(\bullet)$ refers to the cardinality, such that at least one coordinate of $\text{support}(\widetilde \psi_{j, k_1, G})$ and $ \text{support}(\widetilde \psi_{j, k_2, G})$ is empty for all $k_1\neq k_2 \in \{r, \cdots,  2^j - 1-M\}^d$. In addition, we have $\text{Card}(\calI) \lesssim 2^{jd}$ by definition.
	
	Then for any $Q^\pi \in \Lambda_{\infty}(p, L)$ such that $\norm{Q^\pi}_{\Lambda_{\infty}^p} \leq L/2$, where $\norm{\bullet}_{\Lambda_{\infty}^p}$ is the Besov norm and for each $i \in \calI$, define
	$$
	Q^\pi_i = Q^\pi + c_02^{-j(p+d/2)}\widetilde\psi_{j, i, G}.
	$$
	Correspondingly, for every $(s, a, s')$, let
	$$
	h^\pi_i(s, a, s') = h^\pi_0(s, a, s') + c_02^{-j(p+d/2)}\left(\widetilde\psi_{j, i, G}(s, a) - \gamma\int_{a' \in \calA}\pi(a' | s')\widetilde\psi_{j, i, G}(s', a')\text{d}a'\right),
	$$
	where $c_0$ is some positive constant specified later. It can be seen that  for all $i \in \calI$,
	$$
	\norm{c_02^{-j(p+d/2)}\widetilde\psi_{j, i, G}(s, a) }_{\Lambda^p_{\infty}} \lesssim c_0.
	$$
	Hence $\norm{Q_i^\pi}_{\Lambda^p_{\infty}} \leq L$ for sufficient small $c_0$.
	
	For $i \in \{ 0 \}^d \cup \calI$, consider Model \eqref{Model: gaussian model} with the true function $h^\pi_i$ and define the joint probability density of $\{S_{j, t}, A_{j, t}, R_{j, t}, S'_{j, t}\}_{1 \leq j \leq N, 0 \leq t \leq T-1}$ as $P_i$ such that
	$$
	P_i = \Pi_{j=0}^{N}\Pi_{t =0}^{T-1} d^{\nu}(S_{j, 0}, A_{j, 0})p_{h_i^\pi}(R_{j, t} | S_{j, t}, A_{j, t})q(S'_{j, t} \given S_{j, t}, A_{j, t}),
	$$
	by recalling that they are $i.i.d.$ samples, where $p_{h_i^\pi}$ denotes the conditional density of reward given a state-action pair. In particular, when $i = \{0\}^d$, $Q^\pi_i = Q^\pi$ and $h^\pi_i= h^\pi$.
	
	First of all, for sufficiently small $c_0$, we can show $$\norm{c_02^{-j(p+d/2)}\widetilde \psi_{j, i, G}}_{\Lambda^p_{\infty}} \lesssim c_0 \leq L$$ for every $i \in \calI$. In addition, by Equation \eqref{eqn: well-posed operator}, we have
	\begin{align}
	\sqrt{p_{\min}/p_{\max}}(1-\gamma)c_02^{-j(p+d/2)} &\leq \norm{\calT c_02^{-j(p+d/2)}\left(\widetilde \psi_{j, i, G}(S, A) - \gamma\int_{a' \in \calA}\pi(a' | S')\widetilde\psi_{j, i, G}(S', a')\text{d}a'\right)}_{2}\\
	& \lesssim c_02^{-j(p+d/2)}.
	\end{align}
	Secondly, for every $i \in \calI$, the Kullback-Leibler distance $K(P_i, P_0)$ can be bounded as
	\begin{align}
	K(P_i, P_0) &\leq \frac{1}{2}c^2_02^{-j(2p+d)}\sum_{m = 1}^N \sum_{t = 0}^{T-1} \EE\left[\frac{\left(\calT\left(\widetilde\psi_{j, i, G}(S_{m, t}, A_{m, t}) - \int_{a' \in \calA}\pi(a' | S'_{m, t})\widetilde \psi_{m, i, G}(S'_{m, t}, a')\text{d}a'\right)\right)^2}{\sigma^2(S_{m, t}, A_{m, t})}\right]\\
	& \lesssim NTc^2_02^{-j(2p+d)},
	\end{align}
	by the condition in Theorem \ref{thm: lower bound}.
	By choosing $2^j \asymp (NT/\log(NT))^{1/(2p+d)}$, it gives that
	$$
	K(P_i, P_0) \lesssim c_0^2 \log(NT),
	$$
	and $\log(\text{Card}(\calI)) \gtrsim j \gtrsim \log(NT) - \log\log(NT)$. So for sufficiently small $c_0$ and large $NT$, $K(P_i, P_0) \leq 1/8\log(\text{Card}(\calI))$ for every $i \in \calI$. 
	
	Lastly, it can be seen that for $i_1, i_2 \in \calI$ and $i_1 \neq i_2$,
	\begin{align*}
	\norm{\partial^\alpha Q^\pi_{i_1} - \partial^\alpha Q^\pi_{i_2}}_\infty & = c_02^{-j(p+d/2)}\norm{\partial^\alpha \widetilde \psi_{j, {i_1}, G} - \partial^\alpha \widetilde \psi_{j, {i_2}, G}}_\infty\\
	& \gtrsim 2c_02^{-j(p+d/2)}2^{jd/2}2^{j\norm{\alpha}_{\ell_1}}\norm{ \psi^{|\alpha|}}_\infty\\&  =2c_02^{-j(p-\norm{\alpha}_{\ell_1})}\norm{\psi^{|\alpha|}}_\infty,
	\end{align*}
	where the first inequality is given by recalling that at least one coordinate of $\text{support}(\widetilde \psi_{j, k_1, G})$ and $ \text{support}(\widetilde \psi_{j, k_2, G})$ is empty for all $k_1\neq k_2 \in \{r, \cdots,  2^j - 1-M\}^d$. Here $\psi^{|\alpha|}$ refers to $\Pi_{m = 1}^d\partial^{\alpha_m}\psi$.
	
	Note that $2^{-j(p - \norm{\alpha}_{\ell_1})} = (\log(NT)/NT)^{(p - \norm{\alpha}_{\ell_1})/(2p+d)}$. Then Theorem 2.5 of \cite{Tsybakov2009} implies that for any $0 \leq \norm{\alpha}_{\ell_1} < p$,
	\begin{align}
	\liminf_{NT \rightarrow \infty}\inf_{\wh Q} \sup_{Q \in \Lambda_\infty(p, L)}\text{Pr}^Q\left(\norm{\partial^\alpha\wh Q - \partial^\alpha Q}_\infty \geq c (\log(NT)/NT)^{(p - \norm{\alpha}_{\ell_1})/(2p+d)}\right) \geq c' >0,
	\end{align}
	for some constants $c$ and $c'$.
	
	\subsection{Lower bounds for $L^2$-norm Rates}
	The proof mainly follows that of Theorem G.3 in \cite{chen2013optimal}. Again we focus on Model \eqref{Model: gaussian model} and apply Theorem 2.5 of \cite{Tsybakov2009}.
	
	We restrict $\calS \times \calA = [0, 1]^d$.  Let $\{\widetilde \phi_{ j, k, G}, \widetilde\psi_{j, k, G} \}_{j, k, G}$ be a tensor-product wavelet basis of regularity larger than $p$ for $L^2([0, 1]^d)$, where $j$ is the resolution level, $k = (k_1, k_2, \cdots, k_d) \in \{0, 1, \cdots, 2^{j} - 1\}^d$, and $G$ is a vector indicating which element in a Daubechies pair $\{\phi, \psi \}$ is used. Note that $\phi$ has support $[-M+1, M]$ for some positive integer $M$. All these pairs are generated by CDV wavelets \citep{cohen1993wavelets}. Following the proof of \cite{chen2013optimal}, we consider a class of submodels around $Q^\pi$. In particular, for a given $j$, consider the wavelet space $(\calS \times \calA)_j$, which consists of $2^{jd}$ functions $\{\widetilde \psi_{j, k, G}\}_{k \in \{0, \cdots, 2^j - 1\}^d }$ with $G$ chosen as all $\psi$ functions. For some constant $r$, consider $\{\widetilde \psi_{j, k, G}\}_{k \in \{r, \cdots,  2^j - 1-M\}^d}$ as interior wavelets and $\widetilde \psi_{j, k, G}(s, a) = \Pi_{m=1}^{d-1} \psi_{j, k_m}(s_m)\psi_{j, k_d}(a)$ for $k = (k_1, \cdots, k_d) \in \{r, \cdots,  2^j - 1-M\}^d$, where $s = (s_1, \cdots, s_{d-1}) \in \calS$, $a \in \calA$ and $\psi_{j, k_m}(\bullet) = 2^{j/2}\psi(2^j(\bullet) - k_m)$ for $1 \leq m \leq d$.
	Then for sufficiently large $j$, there exists a set $\calI \subseteq \{r, \cdots, (2^j - M - 1)\}^d$ of interior wavelets with $\text{Card}(\calI) \gtrsim 2^{dj}$, where $\text{Card}(\bullet)$ refers to the cardinality, such that at least one coordinate of $\text{support}(\widetilde \psi_{j, k_1, G})$ and $ \text{support}(\widetilde \psi_{j, k_2, G})$ is empty for all $k_1\neq k_2 \in \{r, \cdots,  2^j - 1-M\}^d$. In addition, we have $\text{Card}(\calI) \lesssim 2^{jd}$ by definition.
	
	Then for any $Q^\pi \in \Lambda_{2}(p, L/2)$, where $\norm{\bullet}_{\Lambda_{2,2}^p}$ is the Sobolev norm with smoothness $p$. For each $\theta = \{\theta_i \}_{i \in \calI}$, where $\theta_i \in \{0, 1\}$, define
	$$
	Q^\pi_{\theta} = Q^\pi + c_02^{-j(p+d/2)}\sum_{i \in \calI}\theta_i\widetilde\psi_{j, i, G}(s, a),
	$$
	Correspondingly, let
	$$
	h^\pi_{\theta}(s, a, s') = h^\pi_0(s, a, s') + c_02^{-j(p+d/2)}\left(\sum_{i \in \calI}\theta_i \widetilde\psi_{j, i, G}(s, a) - \gamma\int_{a' \in \calA}\pi(a' | s')\sum_{i \in \calI}\theta_i\widetilde\psi_{j, i, G}(s', a')\text{d}a'\right),
	$$
	where $c_0$ is some positive constant specified later. Based on the construction, there are  $2^{\text{Card}(\calI)}$ combinations of $\theta$. It can be seen that  for every $\theta$,
	\begin{align*}
		&\norm{c_02^{-j(p+d/2)}\sum_{i \in \calI}\theta_i\widetilde\psi_{j, i, G}(\bullet, \bullet) }_{\Lambda^p_{2, 2}} \\
		\lesssim & c_02^{-j(p+d/2)} \sqrt{\sum_{i \in \calI}\theta_i^2 2^{2jp}}\\
		\leq & c_0
	\end{align*}
	Hence $\norm{Q_{\theta}^\pi}_{\Lambda^p_{2, 2}} \leq L$ for sufficient small $c_0$.
	
	First of all, it can be seen that for every $\theta_1$ and $\theta_2$,
	\begin{align*}
	\norm{\partial^\alpha Q^\pi_{\theta_1} - \partial^\alpha Q^\pi_{\theta_2}}_2 & = c_02^{-j(p+d/2 - \norm{\alpha}_{\ell_1})}\norm{\sum_{i \in \calI}(\theta_{1, i} - \theta_{2, i})2^{j/2}\psi^{|\alpha|}(2^j \bullet - i)}_2\\
	& \gtrsim c_02^{-j(p+d/2-\norm{\alpha}_{\ell_1})}\sqrt{\sum_{i \in \calI}(\theta_{1, i} - \theta_{2, i})^2\norm{2^{j/2}\psi^{|\alpha|}(2^j \bullet - i)}^2_2}\\
	& \asymp  2c_02^{-j(p+d/2 - \norm{\alpha}_{\ell_1})}\sqrt{\sum_{i \in \calI}\mathbb{I}(\theta_{1, i} \neq \theta_{2, i}) },
	\end{align*}
	where the second inequality is given by recalling that at least one coordinate of $\text{support}(\widetilde \psi_{j, k_1, G})$ and $ \text{support}(\widetilde \psi_{j, k_2, G})$ is empty for all $k_1\neq k_2 \in \{r, \cdots,  2^j - 1-M\}^d$. Here $\psi^{|\alpha|}(2^j \bullet - i) =\Pi_{m=1}^d\partial^{\alpha_m}\psi(2^j \bullet - i_m).$ The last line is based on that $\widetilde \psi_{j, i, G} \in C^{\omega}$ with $\omega > p > \norm{\alpha}_{\ell_1}$ and is compactly supported with the bounded above and below density, then $\norm{2^{j/2}\psi^{(\alpha)}(2^j \bullet - i)}_2  \asymp 1$. Take $j$ large enough. By Varshamov-Gilbert bound, we can show that there exists a subset of $\{\theta^{(0)}, \cdots, \theta^{(I^\ast)} \}$ such that $\theta^{(0)} = \{0\}^{\text{Card}(\calI)}$, $I^\ast \asymp 2^{\text{Card}(\calI)}$, and $\sqrt{\sum_{j \in \calI}\mathbb{I}(\theta^{(i)}_j \neq \theta^{(k)}_j) } \gtrsim 2^{jd/2}$, where $0 \leq i \leq k \leq I^\ast$. Therefore $\norm{Q^\pi_{i} - Q^\pi_{k}}_2 \gtrsim c_0 2^{-j(p - \norm{\alpha}_{\ell_1})}$ for $0 \leq i \leq k \leq I^\ast$, where we denote $Q_{\theta^{(i)}} = Q_i$. Similarly, we denote $h^\pi_{\theta^{(i)}} = h^\pi_i$.
	
	For $0 \leq m \leq I^\ast$, consider Model \eqref{Model: gaussian model} with the true function $h^\pi_m$ and define the joint probability distribution of $\{S_{j, t}, A_{j, t}, R_{j, t}, S'_{j, t}\}_{1 \leq j \leq N, 0 \leq t \leq T-1}$ as $P_i$ such that
	$$
	P_m = \Pi_{j=0}^{N}\Pi_{t =0}^{T-1} d^{\nu}(S_{j, 0}, A_{j, 0})p_{h_m^\pi}(R_{j, t} | S_{j, t}, A_{j, t})q(S'_{j, t} \given S_{j, t}, A_{j, t}),
	$$
	by recalling that they are $i.i.d.$ samples.
	
	Secondly, for sufficiently small $c_0$, we can show for every $0 \leq m \leq I^\ast$ $$\norm{c_02^{-j(p+d/2)}\sum_{i \in \calI}\theta^{(m)}_i\widetilde\psi_{j, i, G} }_{\Lambda^p_{2, 2}} \lesssim c_0 \leq L/2$$. In addition, by Equation \eqref{eqn: well-posed operator}, we have
	\begin{align*}
	\sqrt{p_{\min}/p_{\max}}(1-\gamma)c_02^{-j(p+d/2)}\norm{\sum_{i \in \calI}\theta^{(m)}_i\widetilde\psi_{j, i, G}}_2 &\lesssim \norm{\calT c_02^{-j(p+d/2)}\left(\sum_{i \in \calI}\theta^{(m)}_i\widetilde\psi_{j, i, G} (S, A) \right.\\
		&\left.-  \gamma\int_{a' \in \calA}\pi(a' | S')\sum_{i \in \calI}\theta^{(m)}_i\widetilde\psi_{j, i, G} (S', a')\text{d}a'\right)}_{2}\\
	& \lesssim c_02^{-j(p+d/2)}\norm{\sum_{i \in \calI}\theta^{(m)}_i\widetilde\psi_{j, i, G}}_2\\
	& \lesssim c_02^{-j(p+d/2)}\sqrt{\sum_{i \in \calI}(\theta^{(m)})^2}\\
	& \asymp c_02^{-jp}.
	\end{align*}
Moreover, for every $0 \leq m \leq I^\ast$, the distance $K(P_m, P_0)$ can be bounded as
	\begin{align*}
	&K(P_m, P_0) \\
	\leq& \frac{1}{2}c^2_02^{-j(2p+d)}\sum_{k = 1}^N \sum_{t = 0}^{T-1} \EE[\frac{(\calT(\sum_{i \in \calI}\theta^{(m)}_i\widetilde\psi_{j, i, G} (S_{k, t}, A_{k, t}) - \int_{a' \in \calA}\pi(a' | S'_{k, t})\sum_{i \in \calI}\theta^{(m)}_i\widetilde\psi_{j, i, G} (S'_{k, t}, a')\text{d}a'))^2}{\sigma^2(S_{k, t}, A_{k, t})}]\\
	\lesssim &  NTc^2_02^{-j(2p)},
	\end{align*}
	by the condition in Theorem \ref{thm: lower bound}.
	By choosing $2^j \asymp (NT)^{1/(2p+d)}$, it gives that
	$$
	K(P_m, P_0) \lesssim c_0^2 (NT)^{d/(2p+d)},
	$$
	and $\log(I^\ast) \gtrsim 2^{jd} \asymp (NT)^{d/(2p+d)}$ by recalling that $I^\ast \asymp 2^{\text{Card}(\calI)}$ and $\text{Card}(\calI)\asymp 2^{jd}$. So for sufficiently small $c_0$ and large $NT$, $K(P_m, P_0) \leq 1/8\log(\calI^\ast)$ for every $1 \leq m \in \calI^\ast$.

	Note that $2^{-j(p-\norm{\alpha}_{\ell_1})} = (NT)^{(\norm{\alpha}_{\ell_1} - p)/(2p+d)}$. Then Theorem 2.5 of \cite{Tsybakov2009} implies that
	\begin{align}
	\liminf_{NT \rightarrow \infty}\inf_{\wh Q} \sup_{Q \in \Lambda_2(p, L)}\text{Pr}^Q\left(\norm{\wh Q - Q}_2 \geq \bar c (NT)^{(\norm{\alpha}_{\ell_1}- p)/(2p+d)}\right) \geq \bar c' >0,
	\end{align}
	for some constants $\bar c$ and $\bar c'$.

\section{Proof of Theorem \ref{thm: sup-norm rate}}

	Let $Q^\pi_{0,J}$ solves $\inf_{Q \in \Psi_J}\norm{Q^\pi - Q}_\infty$. Under all assumptions in Lemmas \ref{lm: variance+bias bound}, we have as long as $\zeta^2\sqrt{\log(J)\log (NT)/NT}= O(1)$,
	\begin{align*}
	&\|\wh Q^\pi - Q^\pi\|_\infty\\
	\leq & \|\wh Q^\pi - \widetilde{Q}^\pi \|_\infty + \|\widetilde{Q}^\pi- Q^\pi \|_\infty\\
	\leq & O_p \Big(\frac{R_{\max}}{1-\gamma}\tau_{J} \xi_{J} \sqrt{(\log J)/(NT e_J)}\Big) + O_p(1) \times \|Q^\pi- \overline{\Pi}Q^\pi \|_\infty\\
	\leq & O_p \Big( \xi_{J} \sqrt{(\log J)/(NT e_J)}\Big) + O_p(1)\norm{\overline \Pi_J}_\infty\|Q^\pi- Q_{0, J}^\pi \|_\infty,
	\end{align*}
	where the first inequality is by triangle inequality and the last inequality is given by Lebesgue's lemma and $\tau_J \lesssim 1$. 
	
	To proceed our proof, we only consider the wavelet basis of \cite{cohen1993wavelets} for $B_K$ and $\Psi_J$, while results of other bases given in Theorem \ref{thm: sup-norm rate} can be derived similarly. Based on the property of wavelet basis, we can show that $\norm{\overline \Pi_J}_\infty \lesssim 1$, where the proof is given in \cite{chen2015optimal}.  Since $Q^\pi \in \Lambda_\infty(p, L)$, we have $\|Q^\pi- Q_{0, J}^\pi \|_\infty \lesssim O(J^{-p/d})$ by e.g., \cite{huang1998projection}. Summarizing together, we have
	$$
	\|\wh Q^\pi - Q^\pi\|_\infty = O_p \Big(\tau_{J} \xi_J \sqrt{(\log J)/(NT e_J)} + J^{-p/d}\Big).
	$$
	According to Lemma \ref{coro: sieve well-posed} and by the property of wavelet bases, $e_J \gtrsim (1-\gamma)^2p^2_{\min}/p_{\max} \gtrsim 1$. Similarly, we can show that $\zeta_b \leq \zeta \lesssim \sqrt{J}$, and $\xi_J \lesssim \sqrt{J}$. Hence we have our first statement that
		\begin{equation}\label{eqn: h error}
		\| \wh Q^\pi - Q^\pi \|_{\infty} = O_p \big(J^{-p/d} +  \sqrt{J(\log J)/(NT)}\big),
		\end{equation}
		as long as $J\sqrt{\log(J)\log (NT)/NT} = O(1)$.
	Lastly, by choosing $J\asymp \left(\frac{NT}{\log(NT)} \right)^{d/(2p+d)}$, which satisfies the constraint, we have
	$$
	\|\wh Q^\pi - Q^\pi\|_\infty  = O_p\left(\left(\frac{\log(NT)}{NT} \right)^{p/(2p+d)} \right).
	$$
	
	Next, we present the proof related to the derivative case. Note that by the previous result, we have
	$$
	\|Q^\pi - \widetilde{Q}^\pi \|_\infty =O_p(J^{-p/d}).
	$$
	In addition, by  Bernstein inequalities in approximation theory, we have
	$$
	\norm{\partial^\alpha Q}_\infty = O(J^{\norm{\alpha}_{\ell_1}/d}) \norm{Q}_\infty,
	$$
	for all $Q \in \Psi_J$. Hence we can show that by Lemma \ref{lm: contraction} and \ref{lm: variance+bias bound} Result (2),
	\begin{align*}
	\norm{\partial^\alpha \widetilde{Q}^\pi - \partial^\alpha Q^\pi}_\infty & \leq \norm{\partial^\alpha \widetilde{Q}^\pi - \partial^\alpha (\overline {\Pi}_J Q^\pi)}_\infty + \norm{\partial^\alpha Q^\pi - \partial^\alpha (\overline {\Pi}_J Q^\pi)}_\infty\\
	& \leq O(J^{\norm{\alpha}_{\ell_1}/d}) \norm{\widetilde{Q}^\pi - \overline{\Pi}_J Q^\pi}_\infty + \norm{\partial^\alpha Q^\pi - \partial^\alpha (\overline{\Pi}_J Q^\pi)}_\infty\\
	& \leq O(J^{\norm{\alpha}_{\ell_1}/d}) \norm{\widetilde{h}^\pi - \Pi_J h_0^\pi}_\infty + \norm{\partial^\alpha Q^\pi - \partial^\alpha (\overline{\Pi}_J Q^\pi)}_\infty\\
	& \leq O_p(J^{-(p -\norm{\alpha}_{\ell_1})/d}) + \norm{\partial^\alpha Q^\pi - \partial^\alpha Q^\pi_J}_\infty + \norm{\partial^\alpha Q^\pi_J - \partial^\alpha (\overline {\Pi}_J Q^\pi)}_\infty\\
	& \leq O_p(J^{-(p -\norm{\alpha}_{\ell_1})/d}) + \norm{\partial^\alpha Q^\pi - \partial^\alpha Q^\pi_J}_\infty + O(J^{\norm{\alpha}_{\ell_1}/d})\norm{Q^\pi_J -  Q^\pi}_\infty.
	\end{align*}
	By choosing $Q^\pi_J$ such that $\norm{Q^\pi_J -   Q^\pi}_\infty = O(J^{-p/d})$ and $\norm{\partial^\alpha Q^\pi_J -  \partial^\alpha (\overline{\Pi}_J Q^\pi)}_\infty = O(J^{-(p - \norm{\alpha}_{\ell_1})/d})$, we have
	$$
	\norm{\partial^\alpha \widetilde{Q}^\pi - \partial^\alpha Q^\pi}_\infty = O_p(J^{-(p - \norm{\alpha}_{\ell_1})/d}).
	$$
	Finally, we can derive that
	\begin{align*}
	\norm{\partial^\alpha \widehat{Q}^\pi - \partial^\alpha {Q}^\pi}_\infty &\leq \norm{\partial^\alpha \widehat{Q}^\pi - \partial^\alpha \widetilde{Q}^\pi}_\infty + \norm{\partial^\alpha \widetilde{Q}^\pi - \partial^\alpha {Q}^\pi}_\infty\\
	&\leq O(J^{\norm{\alpha}_{\ell_1}/d})\norm{\widehat{Q}^\pi - \widetilde{Q}^\pi}_\infty + \norm{\partial^\alpha \widetilde{Q}^\pi - \partial^\alpha {Q}^\pi}_\infty\\
	&\leq O(J^{\norm{\alpha}_{\ell_1}/d})O_p \Big(\xi_J \sqrt{(\log J)/(NT)}\Big)+ O_p(J^{-(p - \norm{\alpha}_{\ell_1})/d}),
	\end{align*}
	where the last inequality is given by Lemma \ref{lm: variance+bias bound}~(1). This concludes our proof.
	%Combining with Lemma \ref{lm: contraction}, we obtain the sup-norm rate. Under one additional assumption, i.e., Assumption \ref{ass: approx}~(c), we can show that 
	%$$
	%\|\wh h^\pi - h_0^\pi\|_\infty \leq O_p \Big(\frac{R_{\max}}{1-\gamma}\tau_{J} \xi_\kappa^\pi \sqrt{(\log J)/(NT e_J)}\Big) + O_p( \frac{R_{\max}}{1-\gamma}J^{-p/d}),
	%$$
	%which implies that 
	%$$
	%\|\wh h^\pi - h_0^\pi\|_\infty  = O_p\left(\frac{R_{\max}p_{\max}\sqrt{(1+\frac{p_{\max}\gamma^2}{p_{\min}})}}{(1-\gamma)}\left(\frac{\log(NT)}{NT} \right)^{p/(2p+d)} \right),
	%$$
	%which completes our proof for the second sup-norm rate by again incorporating Lemma \ref{lm: contraction}.
	\section{Proof of Lemma \ref{lm: variance+bias bound} Result (1)}\label{app: chat}
	The proof consists of three steps. 
	
	\textbf{Step 1}: Decompose the difference between $\wh c$ and $\widetilde c$. 
	\begin{align*}
	\wh c - \widetilde c & =  [\Gamma_\pi^\top B(B^\top B)^-B^\top\Gamma_\pi]^- \Gamma_\pi^ \top B (B^\top B)^- B^\top (\mathbf R - H_0)\\
	& = [\Sigma^{\pi \, \top} G_b^{-1} \Sigma^\pi]^{-1} \Sigma^{\pi \, \top}  G_b^{-1} B^{\top}(\frac{\mathbf R - H_0}{NT})\\
	& + \left(-[\Sigma^{\pi \, \top} G_b^{-1} \Sigma^\pi]^{-1} \Sigma^{\pi \, \top}  G_b^{-1} + [\wh \Sigma^{\pi \, \top} \wh G_b^{-} \wh \Sigma^\pi]^{-} \wh \Sigma^{\pi \, \top}  \wh G_b^{-}\right)B^{\top}(\frac{\mathbf R - H_0}{NT})\\
	& = (I) + (II),
	\end{align*}
	where
	$$
	\wh \Sigma^\pi = \frac{B^\top \Gamma_\pi}{NT} \quad \text{and} \quad \wh G_b = \frac{B^\top B}{NT}.
	$$
	
	\textbf{Step 2}: Bound the first term $(I)$. Define an event
	$$
	\calE_{NT} = \left\{\left\norm{\frac{[ G_b]^{-1/2}B^{\top}B[ G_b]^{-1/2}}{NT} - I_K\right} \leq \frac{1}{2}\right\},
	$$
	where $I_K$ is an identity matrix with size $K$. By Lemma \ref{lem-matl2}~(b), we have
	$$
	\left\norm{\frac{[ G_b]^{-1/2}B^{\top}B[ G_b]^{-1/2}}{NT} - I_K\right} = O_p(\zeta_{b}\sqrt{\log(NT)\log(K)/(NT)}),
	$$
	as long as $\zeta\sqrt{\log(NT)\log(K)/(NT)} = o(1)$. Hence we obtain that $\Pr(\calE_{NT}^c) = o(1)$ by the assumption in Lemma \ref{lm: variance+bias bound} that $\zeta^2\sqrt{\log(NT)\log(K)/(NT)} = O(1)$ and $\zeta \geq \sqrt{J}$.

	%By Lemma E.4 of \citep{chen2013optimal}, we can show $\lambda_{\min}(G_b) \gtrsim 1$ under Assumptions \ref{ass: L2 well-posed} and \ref{ass: basis functions}. In addition, by the property of CDV wavelet bases, we can show $\zeta_{b, K} \leq \sqrt{K}$. Then by Lemma 2.2 of \citep{chen2015optimal}, Assumptions \ref{ass: L2 well-posed} and \ref{ass: stationary}, we can show that
	%$$
	%\left\norm{\frac{[ G_b]^{-1/2}B^{\top}B[ G_b]^{-1/2}}{NT} - I_K\right} = o_p(1),
	%$$
	%as long as $\zeta \sqrt{\log(NT)\log(K)/NT} = o(1)$. 
	
	Now for any $x >0$, we can show that
	\begin{align}
	&P(\norm{(I)}_{\ell_\infty} > x)  \\ 
	\leq & \sum_{j = 1}^JP\left(\mid \frac{1}{NT}\sum_{i = 1}^N\sum_{t=0}^{T-1}q^\pi_j(S_{i, t}, A_{i, t})\left(R_{i, t} - h_0^\pi(S_{i, t}, A_{i, t}, S_{i, t+1})\right)\mid > x, \calE_{NT} \right) + \Pr(\calE_{NT}^c),
	\end{align}
	where $q^\pi_j(S_{i, t}, A_{i, t}) = \left\{[\Sigma^{\pi \, \top} G_b^{-1} \Sigma^\pi]^{-1} \Sigma^{\pi \, \top}  G_b^{-1}b^K(S_{i, t}, A_{i, t})\right\}_{j}$ ($j$-th element of a vector).
	Note that 
	$$\EE\left[R_{i, t} - h_0^\pi(S_{i, t}, A_{i, t}, S_{i, t+1}) \mid S_{i, t}, A_{i, t}\right] = 0,
	$$ 
	by the Bellman equation \eqref{eq: Bellman equation for Q}.
	Therefore the sequence $$\{q^\pi_j(S_{i, t}, A_{i, t})\left(R_{i, t} - h_0^\pi(S_{i, t}, A_{i, t}, S_{i, t+1})\right)\}_{ 0 \leq t \leq (T-1), 1 \leq i \leq N}$$ forms a mean $0$ martingale. We aim to apply Freedman's inequality. Firstly, by Assumption \ref{ass: reward} on the reward, we have
	$$
	|R_{i, t} - h_0^\pi(S_{i, t}, A_{i, t}, S_{i, t+1})| \leq \frac{2R_{\max}}{1- \gamma}.
	$$
	In addition, we can show that
	\begin{align*}
	&|q^\pi_j(S_{i, t}, A_{i, t})| \\
	\leq &\norm{[\Sigma^{\pi \, \top} G_b^{-1} \Sigma^\pi]^{-1} \Sigma^{\pi \, \top}  G_b^{-1}b^K(S_{i, t}, A_{i, t})}_{\ell_2}\\
	\leq & \norm{[G^\pi_\kappa]^{-1/2}}_{\ell_2} \norm{[[G^\pi_\kappa]^{-1/2}\Sigma^{\pi \, \top } G_b^{-1} \Sigma^\pi [G^\pi_\kappa]^{-1/2}]^{-1}[G^\pi_\kappa]^{-1/2}\Sigma^{\pi \, \top } G_b^{-1/2}}_{\ell_2} \norm{ G_b^{-1/2}b^K(S_{i, t}, A_{i, t})}_{\ell_2}\\
	\leq & \frac{\zeta}{s_{JK}\sqrt{e_J}},
	\end{align*}
	where
	$$
	s_{JK}^{-1} = \sup_{h \in \Theta^\pi_J} \frac{\norm{h}_{L^2(S, A, S')}}{\norm{\Pi_K \calT h}_{L^2(S, A)}} = s_{\min}(G_b^{-\frac{1}{2}}\Sigma_\pi [G^\pi_\kappa]^{-1/2}]),
	$$
	and $s_{\min}$ refers to the minimum singular value. One can show that $s_{JK}^{-1} \leq \tau_J \lesssim 1$ by Lemma A.1 of \citep{chen2013optimal} using Assumption \ref{ass: approx}~(a)
	
	Secondly, we can show that conditioning on $\calE_{NT}$,
	$$
	\sum_{i = 1}^N \sum_{t=0}^{T-1}\EE\left[\left\{q^\pi_j(S_{i, t}, A_{i, t})\left(R_{i, t} - h_0^\pi(S_{i, t}, A_{i, t}, S_{i, t+1})\right)\right\}^2\mid S_{i, t}, A_{i, t}\right]\leq \frac{6NTR_{\max}^2}{(1-\gamma)^2s^2_{JK}e_J}.
	$$
	This relies on the following argument. Conditioning on $\calE_{NT}$, for every $j$,
	\begin{align*}
	&\sum_{i = 1}^N \sum_{t=0}^{T-1}\EE\left[(q^\pi_j(S_{i, t}, A_{i, t}))^2 \given S_{i, t}, A_{i, t} \right]\\
	= & \sum_{i = 1}^N \sum_{t=0}^{T-1}\norm{\{[\Sigma^{\pi \, \top} G_b^{-1} \Sigma^\pi]^{-1} \Sigma^{\pi \, \top} G_b^{-1/2}\}_{j\bullet} G_b^{-1/2}b^K(S_{i, t}, A_{i, t})}^2_{\ell_2}\\
	\leq & \frac{3NT}{2}\norm{\{[\Sigma^{\pi \, \top} G_b^{-1} \Sigma^\pi]^{-1} \Sigma^{\pi \, \top} G_b^{-1/2}\}_{j\bullet}}^2_{\ell_2}\\
	\leq & \frac{3NT}{2}\norm{[\Sigma^{\pi \, \top} G_b^{-1} \Sigma^\pi]^{-1}}^2_{\ell_2}\\
	\leq & \frac{3NT}{2}\norm{[G^\pi_\kappa]^{-1/2}[[G^\pi_\kappa]^{-1/2}\Sigma^{\pi \, \top } G_b^{-1} \Sigma^\pi [G^\pi_\kappa]^{-1/2}]^{-1}[G^\pi_\kappa]^{-1/2}}^2_{\ell_2}\\
	\leq & \frac{3NT}{2s_{JK}^2e_J},
	\end{align*}
	where the first inequality is given by the event $\calE_{NT}$.
	Then by Freedman's inequality (e.g., Theorem 1.1 of \cite{tropp2011freedman}), we can show,
	\begin{align}
	&\norm{(I)}_{\ell_\infty} =O_p\left(\sqrt{\frac{\log J}{NTe_J}}\right),
	\end{align}
	as long as $\zeta \sqrt{\log(J)/(NT)} = o(1)$.
	%provided $\sqrt{\frac{\log J}{NT}} = o(1)$.
	
	%Second, we can show that
	%
	%$$
	%\EE\left[\left\{q^\pi_j(S_{i, t}, A_{i, t})\left(R_{i, t} - h_0(S_{i, t}, A_{i, t}, S_{i, t+1})\right)\right\}^2\right]\leq \frac{4R_{\max}^2}{(1-\gamma)^2s^2_{JK}e_J}.
	%$$
	%Then by Theorem 4.2 of \cite{chen2015optimal}, we can show,
	%\begin{align}
	%&\norm{(I)}_{\ell_\infty} =O_p\left( \frac{R_{\max}}{s_{JK}(1-\gamma)}\sqrt{\frac{\log(NT)\log J}{NTe_J}}\right),
	%\end{align}
	%provided $\zeta_{b, K}\sqrt{\frac{\log J}{NT}} = o(1)$. The tail probability is polynomial in terms of $N$ and $T$. 
	%However, this seems not sharp in terms of rate. If we use the Markov inequality, we can first show that
	%$$
	%\EE\left[\left\{\sum_{i = 1}^N\sum_{t=0}^{T-1}q^\pi_j(S_{i, t}, A_{i, t})\left(R_{i, t} - h_0(S_{i, t}, A_{i, t}, S_{i, t+1})\right)\right\}^2\right] \leq \frac{4R_{\max}^2}{(1-\gamma)^2s^2_{JK}e_J}.
	%$$
	%Then by Markov inequality,
	%\begin{align}
	%&\norm{(I)}_{\ell_\infty} =O_p\left( \frac{R_{\max}}{s_{JK}(1-\gamma)}\sqrt{\frac{J}{NTe_J}}\right),
	%\end{align}
	Define 
		$$
		(G_b^{-1/2} \Sigma^\pi)^-_l = \left[[\Sigma^{\pi}]^\top (G_b)^{-1}\Sigma^\pi\right]^{-1}[\Sigma^{\pi}]^\top (G_b)^{-1/2},
		$$
		and similarly for $(\wh G_b^{-1/2} \wh \Sigma^\pi)^-_l$.
	
	\textbf{Step 3:} We bound the second term $(II)$. Relying on Lemmas \ref{lem-SGl2}~(a) and  \ref{lem-BuL2}, we have
	\begin{align}
	&\norm{\text{(II)}}_{\ell_\infty} \\
	\leq & 	\| (G_b^{-1/2} \wh \Sigma^\pi)^-_l\wh G_b^{-1/2}G_b^{1/2} - (G_b^{-1/2} \Sigma^\pi)^-_l \|_{\ell^2}\|G_b^{-1/2}B^\top (R - H_0)/(NT)\|_{\ell^2} \\
	= & O_p \Big(s_{JK}^{-2} \zeta \sqrt{(\log(NT)\log J)/(NT e_J)}\Big)O_p(\frac{R_{\max}}{1-\gamma}\sqrt{\frac{K}{NT}})\\
	= & O_p \Big(\sqrt{\log(J)/(NT e_J)}\Big),
	\end{align}
	by the assumption in Lemma \ref{lm: variance+bias bound}~(1) that $\zeta^2\sqrt{\log (NT)} /\sqrt{NT} = O(1)$ and the fact that $\zeta \geq \sqrt{K}$ and $s_{JK}^{-1} \leq \tau_J \lesssim 1$. This completes the proof of Lemma \ref{lm: variance+bias bound}(1) by noting that $\sup_{s \in \calS, a \in \calA}\norm{\psi^J(s, a)}_{\ell_1} = \xi_{J}$ by definition.

	\section{Proof of Lemma \ref{lm: variance+bias bound} Result (2)}\label{app: approx}
	We first prove the following Lemma.
	\begin{lemma}\label{lm: approx}
		Suppose that $\zeta^2\sqrt{\log(J)\log (NT)} /\sqrt{NT} = O(1)$ and let Assumptions \ref{ass: stationary}-\ref{ass: approx} hold. Then 
		$\|\widetilde h^\pi - \Pi_J h^\pi_0\|_\infty \leq O_p(1) \times \|h^\pi_0 - \Pi_J h^\pi_0\|_{\infty}$. %\\
		%(2) $\|\widetilde Q^\pi - Q^\pi\|_\infty \leq O_p\left( 1 + \|\overline \Pi_J \|_{\infty} \right) \times \|Q^\pi - Q^\pi_{0,J}\|_\infty$,
		%where $Q^\pi_{0,J}$ solves $\inf_{Q \in \Psi_J}\norm{Q^\pi - Q}_\infty$.
	\end{lemma}
	The proof follows similarly as Lemma A.3 of \cite{chen2013optimal}. Note that the difference between $\widetilde{h}^\pi$ and $\Pi_J h^\pi_0$ can be decomposed as
	\begin{align*}
	&\widetilde{h}^\pi(s, a, s') - \Pi_J h_0^\pi(s, a, s') \\
	=& \widetilde{\Pi} (h^\pi_0 - \Pi_J h_0^\pi)(s, a, s')\\
	+	&  (\kappa_\pi^J(s, a, s'))^\top(G_b^{-1/2} \Sigma^\pi)^-_l \{G_b^{-1/2} (B^\top(H_0 - \Gamma_\pi c_J)/(NT) - E[ b^K(S, A)(h_0^\pi(S, A, S') - h^\pi_J(S, A, S'))])\} \\
	+	&  (\kappa_\pi^J(s, a, s'))^\top\{(\wh G_b^{-1/2} \wh \Sigma^\pi)^-_l \wh G_b^{-1/2}G_b^{-1/2} -(G_b^{-1/2} \Sigma^\pi)^-_l\} G_b^{-1/2}B^\top(H_0 - \Gamma_\pi c_J)/(NT)\\
	=& (I) + (II) + (III).
	\end{align*}
	
	For $(I)$, by Assumption \ref{ass: approx}~(b), we can show that $\norm{(I)}_\infty \lesssim \norm{h_0^\pi - \Pi_J h_0^\pi}_\infty$.
	For $(II)$, by Lemma \ref{lem-BHJ}, we can show
	\begin{align}
	\norm{(II)}_\infty &\leq \zeta^\pi_{\kappa,J}s_{JK}^{-1} O_p\big(\zeta_{b, K}\sqrt{\frac{\log(NT)\log(K)}{NT}}\big) \|h^\pi_0 - \Pi_J h^\pi_0\|_{\infty}\\
	& =O_p\big(\zeta^2 \sqrt{\frac{\log(NT)\log(K)}{NT}}\big) \|h^\pi_0 - \Pi_J h^\pi_0\|_{\infty} = O(1)\|h^\pi_0 - \Pi_J h^\pi_0\|_{\infty},
	\end{align}
	where we use the fact that $\zeta \geq \max\{\zeta_{b, K}, \zeta_{\kappa,J}\}$, $s_{JK}^{-1} \leq \tau_J \lesssim 1$ and that $\zeta^2\sqrt{\log(J)\log (NT)} /\sqrt{NT} = O(1)$.
	For $(III)$ term, by Lemma \ref{lem-SGl2}(b), we can show that
	\begin{align*}
	\small
	& \norm{(III)}_\infty \\
	\leq & \zeta_{\kappa,J} \norm{(\wh G_b^{-1/2} \wh \Sigma^\pi)^-_l \wh G_b^{-1/2}G_b^{-1/2} -(G_b^{-1/2} \Sigma^\pi)}_{\ell_2} \norm{G_b^{-1/2}B^\top(H_0 - \Gamma_\pi c_J)/(NT)}_{\ell_2}\\
	\leq & \zeta_{\kappa,J} O_p\big(\zeta\sqrt{ \frac{\log J\log(NT)}{NT}}\big)\{O_p\big(\zeta_{b, K}\sqrt{\frac{\log(NT)\log K}{NT}}\big)\|h^\pi_0 - \Pi_J h^\pi_0\|_{\infty} + \norm{\Pi_K\calT(h^\pi_0 - \Pi_Jh^\pi_0)}_{L^2(S, A)}\} \\
	\leq & \zeta_{\kappa,J} O_p\big(\zeta\sqrt{ \frac{\log J\log(NT)}{NT}}\big)\{O_p\big(\zeta_{b, K}\sqrt{\frac{\log(NT)\log K}{NT}}\big)\|h^\pi_0 - \Pi_J h^\pi_0\|_{\infty} + \norm{(h^\pi_0 - \Pi_Jh^\pi_0)}_{L^2(S, A)}\} \\
	=& O_p\big(\zeta^2 \sqrt{\frac{\log(J)\log(NT)}{NT}}\big) \norm{(h^\pi_0 - \Pi_Jh^\pi_0)}_{L^2(S, A)}\\
	=&   O_p(1)\norm{h^\pi_0 - \Pi_Jh^\pi_0}_{L^2(S, A)},
	\end{align*}
	by the condition that  $\zeta^2\sqrt{\log(J)\log (NT)} /\sqrt{NT} = O(1)$. 
	
	Now, we return to Result (2) of Lemma \ref{lm: variance+bias bound}. By Lemma \ref{lm: contraction}, we can see that 
	\begin{align*}
	&\|\widetilde Q^\pi - Q^\pi\|_\infty \lesssim  \|\widetilde h^\pi - h_0^\pi\|_\infty\\
	\leq & \|\widetilde{h}^\pi - \Pi_Jh_0^\pi \|_\infty + \| h_0^\pi -\Pi_Jh_0^\pi \|_\infty\\
	= & O_p(1)\|h_0^\pi- \Pi_Jh_0^\pi \|_\infty,\\
	\leq& O_p(1) \|Q^\pi- \overline \Pi_J Q^\pi \|_\infty,
	\end{align*}
	which concludes our proof.
	%The statement (2) of Lemma \ref{lm: approx} can be shown by Lebesgue's lemma. See the proof of Lemma A.3 of \citep{chen2013optimal}.

	\section{Proof of Theorem \ref{thm: L2 variance+bias bound} }\label{app sec: L2 variance}
	The idea of proof is similar to that in Lemma \ref{lm: variance+bias bound}~(1) and Theorem \ref{thm: sup-norm rate}. By triangle inequality, we have $\norm{\wh Q^\pi - Q^\pi}_2 \leq \norm{\wh Q^\pi - \widetilde Q^\pi}_2 + \norm{\widetilde Q^\pi - \overline \Pi_J Q^\pi}_2 + \norm{Q^\pi - \overline \Pi_J Q^\pi}_2$. In the following $\textbf{Step 1-3}$, we first bound $\norm{\wh h^\pi - \widetilde h^\pi}_2$ since $\norm{\wh Q^\pi - \widetilde Q^\pi}_2 \lesssim \norm{\wh h^\pi - \widetilde h^\pi}_2$. The last step is to bound $\norm{\widetilde Q^\pi - \overline \Pi_J Q^\pi}_2$.
	
	\textbf{Step 1}: Decompose the difference between $\wh h^\pi(s, a, s')$ and $\widetilde h^\pi(s, a, s')$ as follows. 
	\begin{align*}
	& (\kappa_\pi^J(s, a, s'))^\top\wh c - (\psi^J(s, a))^\top\widetilde c =  (\psi^J(s, a))^\top[\Gamma_\pi^\top B(B^\top B)^-B^\top\Gamma_\pi]^- \Gamma_\pi^ \top B (B^\top B)^- B^\top (\mathbf R - H_0)\\
	 =& (\kappa_\pi^J(s, a, s'))^\top[\Sigma^{\pi \, \top} G_b^{-1} \Sigma^\pi]^{-1} \Sigma^{\pi \, \top}  G_b^{-1} B^{\top}(\frac{\mathbf R - H_0}{NT})\\
	+&  (\kappa_\pi^J(s, a, s'))^\top\left(-[\Sigma^{\pi \, \top} G_b^{-1} \Sigma^\pi]^{-1} \Sigma^{\pi \, \top}  G_b^{-1} + [\wh \Sigma^{\pi \, \top} \wh G_b^{-} \wh \Sigma^\pi]^{-} \wh \Sigma^{\pi \, \top}  \wh G_b^{-}\right)B^{\top}(\frac{\mathbf R - H_0}{NT})\\
	= &  (I) + (II),
	\end{align*}
	where
	$$
	\wh \Sigma^\pi = \frac{B^\top \Gamma_\pi}{NT} \quad \text{and} \quad \wh G_b = \frac{B^\top B}{NT}.
	$$
	
	\textbf{Step 2}: Bound the first term $(I)$. Note that
	\begin{align*}
	    \norm{(I)}_2 &= \norm{(\kappa_\pi^J(\bullet, \bullet, \bullet))^\top[\Sigma^{\pi \, \top} G_b^{-1} \Sigma^\pi]^{-1} \Sigma^{\pi \, \top}  G_b^{-1} B^{\top}(\frac{\mathbf R - H_0}{NT})}_2\\
	    & = \norm{[G^\pi_{\kappa}]^{1/2}[\Sigma^{\pi \, \top} G_b^{-1} \Sigma^\pi]^{-1} \Sigma^{\pi \, \top}  G_b^{-1} B^{\top}(\frac{\mathbf R - H_0}{NT})}_{\ell_2}\\
	    & \leq s^{-1}_{KJ} \norm{ G_b^{-1/2}B^{\top}(\frac{\mathbf R - H_0}{NT})}_{\ell_2} = O_p(\sqrt{\frac{K}{NT}}),
	\end{align*}
	where the last inequality is given by Lemma \ref{lem-BuL2} and $s_{JK}^{-1} \lesssim 1$.
	
	\textbf{Step 3}: we bound the second term $(II)$. Relying on Lemmas \ref{lem-SGl2}~(b) and  \ref{lem-BuL2}, we have
	\begin{align}
	&\norm{\text{(II)}}_{2} \\
	\leq & 	\|[G^\pi_{\kappa}]^{1/2}\{ (G_b^{-1/2} \wh \Sigma^\pi)^-_l\wh G_b^{-1/2}G_b^{1/2} - (G_b^{-1/2} \Sigma^\pi)^-_l\} \|_{\ell^2}\|G_b^{-1/2}B^\top (\mathbf{R} - H_0)/(NT)\|_{\ell^2} \\
	= & O_p \Big(s_{JK}^{-2} \zeta \sqrt{(\log(NT)\log J)/(NT)}\Big)O_p(\frac{R_{\max}}{1-\gamma}\sqrt{\frac{K}{NT}})\\
	= & O_p \Big(\sqrt{K/(NT)}\Big),
	\end{align}
	by the assumption in Theorem \ref{thm: L2 variance+bias bound} that $\zeta\sqrt{\log (NT)\log(J)} /\sqrt{NT}) = o(1)$ and $s_{JK}^{-1} \lesssim 1$.
	
	\textbf{Step 4}: In the remaining proof, we show the bound for $\norm{\widetilde Q^\pi - \overline \Pi_J Q^\pi}_2$. By Theorem \ref{thm: well-posed}, we can show that,
	$$
	(1+\gamma)\norm{\widetilde Q^\pi - \overline \Pi_J Q^\pi}_2 \lesssim \norm{\widetilde h^\pi - \Pi_J h^\pi}_2.
	$$
	Then by a similar proof in Appendix \ref{app: approx}, we can show that  as long as $\zeta\sqrt{\log (NT)\log(J)} /\sqrt{NT}) = O(1)$,
	$$
	\norm{\widetilde Q^\pi - \overline \Pi_J Q^\pi}_2 \leq O_p(1) \times \norm{Q^\pi - \overline \Pi_J Q^\pi}_2 = O_p(1)\times J^{-p/d},
	$$
	where we use the existing result on the approximation error of the linear sieve in the last equation. Summarizing \textbf{Step 1-4} together, we obtain the statements in Theorem \ref{thm: L2 variance+bias bound}.
	Finally, we conclude our proof by the similar argument in the proof of Theorem \ref{thm: sup-norm rate} for the derivatives case.

	\section{Technical Lemmas}
	\begin{lemma}\label{coro: sieve well-posed}
		For any policy $\pi$, under Assumptions \ref{ass: Markovian}, \ref{ass: DGP}-\ref{ass: stationary}, we have %$e_J \gtrsim \frac{p^2_{\min}}{p_{\max}}(1-\gamma)^2$, and 
		$$
		e_J \gtrsim \frac{p^2_{\min}}{p_{\max}}(1-\gamma)^2\omega_J
		$$
		for every $J \geq 1$, where $\omega_J = \lambda_{\min}(\EE\left[\psi^J(S, A) (\psi^J(S, A))^\top \right])$
	\end{lemma}
	\begin{proof}
		By definition,
		$$
		e_J = \lambda_{\min}\left\{\EE\left[\left(\psi^J(S, A) - \gamma \psi_\pi^J(S')\right)\left(\psi^J(S, A) - \gamma \psi_\pi^J(S')\right)^{\top}\right]\right\}.
		$$
		Applying Theorem \ref{thm: well-posed} with $Q_1(s, a) = (\psi^J(S, A))^\top x$ and $Q_2(s, a) = 0$ for every $s \in \calS$ and $a \in \calA$ (recall that the sieve space is a subset of $L^2(S, A)$), we have
		\begin{align*}
		&	x^\top \EE\left[\left(\psi^J(S, A) - \gamma \psi_\pi^J(S')\right)\left(\psi^J(S, A) - \gamma \psi_\pi^J(S')\right)^{\top}\right]x\\
		\geq &	x^\top \EE\left[\left(\psi^J(S, A) - \gamma \EE\left[\psi_\pi^J(S')\mid S, A\right]\right)\left(\psi^J(S, A) - \gamma \EE\left[\psi_\pi^J(S')\mid S, A\right]\right)^{\top}\right]x\\
		= & \norm{(\psi^J(S, A) - \gamma \EE\left[\psi_\pi^J(S')\mid S, A\right])^\top x}^2_{2}\\
		\geq & \frac{p_{\min}}{p_{\max}}(1-\gamma)^2 \norm{(\psi^J(S, A))^\top x}_{2}^2 \geq \frac{p_{\min}}{p_{\max}}(1-\gamma)^2\omega_J\norm{x}_{\ell_2}^2,
		\end{align*}
		where the first inequality is given by Jensen's inequality and the last inequality is by the definition of $\omega_J$. %For technical simplicity, we focus on the case where $\Psi_J$ is spanned by the wavelet basis of \citep{cohen1993wavelets}, while the result for other bases can be similarly derived.
	\end{proof}
	
	By examining the proof, we can see that the above lemma also holds for $\bar d_T^{\pi^b}$ without Assumption \ref{ass: stationary}.
	
	Next we present several technical lemmas adapted from \cite{chen2013optimal}. Define the \emph{orthonormalized} matrix estimators
	\begin{eqnarray*}
		\wh G_b^o & = & G_b^{-1/2} \wh G_b G_b^{-1/2} \\
		\wh G_\kappa^{\pi, \, o} & = & [G_\kappa^\pi]^{-1/2} \wh G_\kappa^\pi [G_\kappa^\pi]^{-1/2} \\
		\wh \Sigma^{\pi, \, o} & = & G_b^{-1/2} \wh \Sigma^{\pi} [G_\kappa^\pi]^{-1/2},
	\end{eqnarray*}
	where $\wh G_\kappa^\pi = \frac{\Gamma_\pi^{\top} \Gamma_\pi}{NT}$.
	Let $G_b^o = I_K$, $G_\kappa^{\pi \, o} = I_J$ and $\Sigma^{\pi \, o}$ denote their corresponding expected values.

	\begin{lemma}\label{lem-matl2}
		Under Assumption \ref{ass: stationary}, the following three bounds hold.
		\begin{eqnarray*}
			\|\wh G_\kappa^{\pi, \, o} - G_\kappa^{\pi, \, o} \|_{\ell^2} & = & O_p (\zeta^\pi_{\kappa,J} \sqrt{( \log (NT)\log(J))/(NT)}) \\
			\|\wh G_b^o - G_b^o \|_{\ell^2} & = & O_p (\zeta_{b,K} \sqrt{(\log (NT) \log K)/(NT)}) \\
			\|\wh \Sigma^{\pi, \, o} - \Sigma^{\pi, \, o}\|_{\ell^2} & = & O_p (\max(\zeta_{b,K}, \zeta^\pi_{\kappa,J})\sqrt{(\log (NT)\log K)/(NT)})\,.
		\end{eqnarray*}
		as $N, T, J,K \to \infty$ as long as $\zeta\sqrt{(\log (NT)\log(J)/NT} = o(1)$.
	\end{lemma}
	\begin{proof}
		The proof follows similarly from Lemma 2.2 of \cite{chen2015optimal}. The basic idea is to use Berbee's coupling lemma (e.g., Theorem 4.2 of \cite{chen2015optimal}) and matrix Bernstein's inequality (e.g., \cite{tropp2015introduction}). For brevity, we only show the proof of the second statement in Lemma \ref{lem-matl2}, while others are similar.
		
		Let $X_{i, t} = G_b^{-1/2} b^{K}(S_{i, t}, A_{i, t})[b^{K}(S_{i, t}, A_{i, t})]^{\top} G_b^{-1/2}/(NT) - I_K/NT$ and $\EE[X_{i, t}] = 0_{K\times K}$. Denote the upper bound of mixing coefficient as $\beta(w) = \beta_0 \exp(-\beta_1 w)$, where $\beta_0$ and $\beta_1$ are given in Assumption \ref{ass: stationary}. By Berbee's lemma, for a fixed $i$ with $1 \leq i \leq N$ and some integer $w$, the stochastic process $\{X_{i, t} \}_{t\geq0}$ can be coupled by a process $Y^\ast_{i, t}$ such that $Y_{i, k} = \{X_{i, (k-1)w+j} \}_{0 \leq j < w}$ and $Y^\ast_{i, k} = \{X^\ast_{i, (k-1)w+j} \}_{0 \leq j < w}$ are identically distributed for each $k \geq 1$ and $P(Y_{i, k} \neq Y^\ast_{i, k}) \leq \beta(w)$. In addition, the sequence $\{\{ Y^\ast_{i, k}\} \,\given \, k = 2z, z \geq 1\}$ are independent and so are the sequence $\{\{ Y^\ast_{i, k}\} \,\given \, k = 2z+1, z \geq 0\}$. Denote $I_{e}$ as the indices of the corresponding even number block and $I_{o}$ as indices of the corresponding odd number blocks in $\{0, \cdots, T-1\}$. Let $I_r$ be the indices in the remainders, i.e., $I_r = \{\floor{T/w}w, \cdots, T-1\}$ and thus $\text{Card}(I_{r}) < w$.  We construct a coupled stochastic process for every $1 \leq i \leq N$ trajectory. Now by triangle inequality, we can show that for $x > 0$
		\begin{align*}
		   & \prob(\norm{\sum_{i = 1}^N \sum_{t = 0}^{T-1}X_{i, t}}_{\ell_2} \geq 4x) \\
		  \leq &  \prob(\norm{\sum_{i = 1}^N \sum_{t = 0}^{\floor{T/w}w-1}X^\ast_{i, t} }_{\ell_2} \geq 2x) + \prob(\norm{\sum_{i =1}^N\sum_{t \in I_{r}} X_{i, t} }_{\ell_2} \geq x)) + \prob(\norm{\sum_{i = 1}^N \sum_{t = 0}^{\floor{T/w}w-1}(X^\ast_{i, t} - X_{i, t})}_{\ell_2}\geq x)\\
		  \leq &  \prob(\norm{\sum_{i = 1}^N \sum_{t \in I_o}X^\ast_{i, t}}_{\ell_2} \geq x) + \prob(\norm{\sum_{i = 1}^N \sum_{t \in I_e}X^\ast_{i, t}}_{\ell_2} \geq x)+\prob(\norm{\sum_{i =1}^N\sum_{t \in I_{r}} X_{i, t} }_{\ell_2} \geq x) + \frac{NT\beta(w)}{w}.
		\end{align*}
		By choosing $w = c\log(NT)$ for sufficiently large $c$, we can show that
		$$
	\frac{NT\beta(w)}{w} \lesssim \frac{1}{NT}.
		$$
		For the term $\prob(\norm{\sum_{i = 1}^N \sum_{t \in I_o}X^\ast_{i, t}}_{\ell_2} \geq x)$, notice that $\sum_{i = 1}^N \sum_{t \in I_o}X^\ast_{i, t}$ has been decomposed into the sum of fewer than $N \times \floor{T}/w$ independent matrices, i.e., $Z^\ast_{i, k} = \sum_{t = (k-1)w }^{kw-1} X^\ast_{i, t}, k \geq 1$. One can show that $\norm{Z^\ast_{i, k}}_{\ell_2} \leq \frac{w(\zeta^2 + 1)}{NT} = w\bar R$ and $\max(\norm{Z^\ast_{i, k}[Z^\ast_{i, k}]^\top}_2,\norm{[Z^\ast_{i, k}]^\top Z^\ast_{i, k}}_2) \leq \frac{w^2(\zeta^2+1)}{(NT)^2} = w^2\sigma^2$. Then by matrix Berstein's inequality, we have
		$$
		\prob(\norm{\sum_{i = 1}^N \sum_{t \in I_o}X^\ast_{i, t} }_{\ell_2}\geq x) \leq 2K\exp\left(\frac{-x^2/2}{(NT)w\sigma^2 +w\bar Rx/3} \right).
		$$
		Then we can bound this probability towards $0$ as $K\rightarrow \infty$ by choosing $x = C\sigma\sqrt{wNT\log(K)}$ for sufficiently large $C$ with the condition given in the statement that $\bar R\sqrt{w\log(K)}=o(\sigma \sqrt{NT})$, i.e., $\zeta \sqrt{\log(NT)\log(K)}/\sqrt{NT} = o(1)$. Similar argument can be applied to $\prob(\norm{\sum_{i = 1}^N \sum_{t \in I_e}X^\ast_{i, t}}_{\ell_2} \geq x)$. 
		
		Next, we derive an upper bound for $\prob(\norm{\sum_{i =1}^N\sum_{t \in I_{r}} X_{i, t} }_{\ell_2} \geq \bar x)$ for some $\bar x > 0$. By Bernstein's inequality and $\sum_{t \in I_{r}} X_{i, t}$ are independent for $1 \leq i \leq N$, we have 
		$$
		\prob(\norm{\sum_{i =1}^N\sum_{t \in I_{r}} X_{i, t} }_{\ell_2} \geq \bar x) \leq 2K \exp\left(\frac{-\bar x^2/2}{Nw^2\sigma^2 + w\bar R \bar x/3}\right).
		$$
		By choosing $\bar x = C_1\sigma\sqrt{NTw\log(K)}$ for sufficiently large $C_1$, we can show that 
		$$
		\prob(\norm{\sum_{i =1}^N\sum_{t \in I_{r}} X_{i, t} }_{\ell_2} \geq \bar x) \lesssim K^{-C_1(T/w)+1},
		$$
		as long as $\zeta \sqrt{\log(NT)\log(K)}/\sqrt{NT} = o(1)$. Without loss of generality, we can assume $w \leq T$, which completes the proof in the second statement of Lemma \ref{lem-matl2} as the probability converges to $0$ as long as $K \rightarrow \infty$. Otherwise the result in the statement can be obtained directly by using the Bernstein's inequality in the i.i.d setting without using Berbee's lemma. Other statements follow similarly.
	\end{proof}

	\begin{lemma}\label{lem-BuL2}
		Under Assumptions \ref{ass: Markovian} and \ref{ass: reward}, $\|G_b^{-1/2}B^\top (\mathbf R - H_0)/(NT)\|_{\ell^2} = O_p(\frac{R_{\max}}{1-\gamma}\sqrt{\frac{K}{NT}})$.
	\end{lemma}
	\begin{proof}
		We apply the Markov inequality. Note that
		\begin{align*}
		& \|G_b^{-1/2}B^\top (\mathbf R - H_0)/(NT)\|^2_{2} \\
		\leq & \frac{4R^2_{\max}}{(1-\gamma)^2} K/NT,
		\end{align*}
		because all the terms in $G_b^{-1/2}B^\top (\mathbf R - H_0)/(NT)$ are uncorrelated by the Bellman equation \eqref{eq: Bellman equation for Q}. Hence the proof completes. 
	\end{proof}
	
	\begin{lemma}\label{lem-BHJ}
		Let $h^\pi_J(s, a, s') = \kappa_\pi^J(s, a, s')^\top c_J$ for any deterministic $c_J \in \mb R^J$ and $H_J = (h^\pi_J(S_{1, 0}, A_{1, 0}, S_{1, 1} ),h^\pi_J(S_{1, 1}, A_{1, 1}, S_{1, 2}), \ldots,h^\pi_J(S_{N, T-1}, A_{N, T-1}, S_{N, T}))^\top=\Gamma_\pi c_J$. Under Assumptions \ref{ass: stationary},
		\begin{eqnarray*}
			& & \|G_b^{-1/2} (B^\top(H_0 - H_J)/(NT) - E[ b^K(S, A)(h_0^\pi(S, A, S') - h^\pi_J(S, A, S'))])\|_{\ell^2} \\
			& & \quad = \quad O_p \left(  \sqrt{\frac{\zeta_{b, K}\log(NT)\log(K)}{NT}} \times \|h_0^\pi - h_J\|_{\infty}  \right) \,.
		\end{eqnarray*}
		provided $\sqrt{\frac{\log(NT)\log(K)}{NT}} = o(1)$.
	\end{lemma}
	\begin{proof}
		We again use Berbee's coupling lemma and matrix Bernstein's inequality (e.g., \cite{dedecker2002maximal,chen2015optimal}) and get the result. The argument is similar to that in the proof of Lemma \ref{lem-matl2}. In particular, let
		$$
		Z_{i, t} = G_b^{-1/2}b^K(S_{i, t}, A_{i, t})(h_0^\pi(S_{i, t}, A_{i, t}, S_{i, t+1}) - h^\pi_J(S_{i, t}, A_{i, t}, S_{i, t+1})).
		$$
		It can be seen that $\norm{Z_{i, t}}_{\ell_2} \leq \zeta_{b, K} \norm{h_0^\pi - h_J}_\infty$ and
		$$
		\max\{\EE[Z_{i, t}^\top Z_{i, t}],  \EE[Z_{i, t} Z_{i, t}^\top]\} \leq \zeta^2_{b, K} \norm{h_0^\pi - h_J}^2_\infty,
		$$
		which gives the result.
	\end{proof}
	\begin{lemma}\label{lem-SGl2}
		Let $s_{JK}^{-1} \zeta \sqrt{(\log(NT)\log J)/(NT)} = o(1)$, and Assumption \ref{ass: stationary} is satisfied.  Then:
		\begin{eqnarray*}
			(a) &  \|(\wh G_b^{-1/2} \wh \Sigma^\pi)^-_l\wh G_b^{-1/2}G_b^{1/2} - (G_b^{-1/2} \Sigma^\pi)^-_l \|_{\ell^2}= O_p \Big(s_{JK}^{-2} \zeta \sqrt{(\log(NT)\log J)/(NT e_J)}\Big) \\
			(b) & \|[G^\pi_{\kappa}]^{1/2}\{(\wh G_b^{-1/2} \wh \Sigma^\pi)^-_l\wh G_b^{-1/2}G_b^{1/2} - (G_b^{-1/2} \Sigma^\pi)^-_l\} \|_{\ell^2} = O_p \Big(s_{JK}^{-2} \zeta \sqrt{(\log(NT)\log J)/(NT))}\Big)
			%(c) & & \|G_b^{-1/2} \Sigma^\pi\{(\wh G_b^{-1/2} \wh \Sigma^\pi)^-_l\wh G_b^{-1/2}G_b^{1/2} - (G_b^{-1/2} \Sigma^\pi)^-_l\} \|_{\ell^2} = O_p \Big(s_{JK}^{-1} \zeta \sqrt{(\log(NT)\log J)/(NT)}\Big)\,,
		\end{eqnarray*}
		where
		$$
		(G_b^{-1/2} \Sigma^\pi)^-_l = \left[[\Sigma^{\pi}]^\top (G_b)^{-1}\Sigma^\pi\right]^{-1}\Sigma^{\pi}]^\top (G_b)^{-1/2},
		$$
		and similarly for $(\wh G_b^{-1/2} \wh \Sigma^\pi)^-_l$.
	\end{lemma}
	\begin{proof}
		We use the similar proof as Lemma F.10 of \cite{chen2013optimal} with Berbee's coupling lemma again.  The argument is also similar to that in the proof of Lemma \ref{lem-matl2}. We omit here for brevity.
	\end{proof}

	\bibliographystyle{agsm}
	\bibliography{reference}
\end{document}